\documentclass[12pt, reqno]{amsart}

\usepackage{amsmath}
\usepackage{amsfonts}
\usepackage{amssymb}
\usepackage{amsthm}
\usepackage[shortlabels]{enumitem}
\usepackage{graphicx}
\usepackage{color}
\usepackage{dsfont}
\usepackage{MnSymbol}
\usepackage{enumitem}
\usepackage{extpfeil}
\usepackage{mathtools}
\usepackage{url}
\usepackage[margin=1in]{geometry}
\usepackage{fancyhdr}
\usepackage[all,cmtip]{xy}

\newtheorem{thmA}{Theorem}

\newtheorem{corA}[thmA]{Corollary}
\newtheorem{propA}[thmA]{Proposition}

\newtheorem{theorem}{Theorem}[section]
\newtheorem{lemma}[theorem]{Lemma}
\newtheorem{prop}[theorem]{Proposition}
\newtheorem{proposition}[theorem]{Proposition}
\newtheorem{cor}[theorem]{Corollary}
\newtheorem{corollary}[theorem]{Corollary}

\theoremstyle{remark}
\newtheorem{remark}[theorem]{Remark}

\newtheorem{example}[theorem]{Example}

\theoremstyle{definition}
\newtheorem{definition}[theorem]{Definition}

\def\X{\mathfrak{X}}
\def\<{\langle}
\def\>{\rangle}
\def\-{\overline}

\def\G{\Gamma}

\def\g{\gamma}

\def\e{\varepsilon}
\def\imm{\rm{im}}
\def\distLX{{\rm{dist}}_L^{\X(F)}}
\def\dist{{\rm{dist}}}
\def\vol{\rm{vol}}

\def\A{{\rm{Aug}}}
\def\P{{\mathcal{P}}}

\def\l{\ell}

\def\f{\delta}

\numberwithin{equation}{section}
\newcommand{\FP}{{\rm{FP}}}

\newcommand{\N}{\mathbb{N}}

\newcommand{\Z}{\mathbb{Z}}

\newcommand{\Q}{\mathbb{Q}}

\begin{document}

\title[Weak commutativity and Dehn functions]
 {Weak commutativity, virtually nilpotent groups, and Dehn functions}  

\author{Martin R. Bridson}

\address{Mathematical Institute,
University of Oxford, 
Andrew Wiles Building, ROQ,
Oxford OX2 6GG,
United Kingdom} 
\email{bridson@maths.ox.ax.uk}

\author{Dessislava H. Kochloukova}

\address
{Department of Mathematics, State University of Campinas (UNICAMP), 13083-859, Campinas, SP, Brazil} 
\email{desi@ime.unicamp.br}
\thanks{The first author was supported in part by a Wolfson Research Merit Award from the Royal Society.
The second author was supported in part by grants 2017/17320-9, 2018/23690-6 from  FAPESP and 401089/2016-9 CNPq, Brazil}



\begin{abstract} The group $\X(G)$ is obtained from $G\ast G$ by forcing
	each element $g$ in the first free factor to commute with the copy of $g$ in the second free factor.  
	We make significant additions to the list of properties that the functor $\X$ is known to  preserve. We also investigate
	the geometry and complexity of the word problem for $\X(G)$. 
	Subtle features of $\X(G)$ are encoded
	in a normal abelian subgroup $W<\X(G)$ that is a module over $\Z Q$, where
	$Q= H_1(G,\Z)$. 
	We establish a structural result for this module and illustrate
	its utility by proving that $\X$ preserves virtual nilpotence, the Engel condition, and growth type -- polynomial,   exponential, or intermediate. We also use it to
	establish  isoperimetric inequalities for $\X(G)$ when $G$ lies in a class that includes 
	Thompson's group $F$ and all non-fibered K\"ahler groups. The word problem is solvable
in $\X(G)$ if and only if it is solvable in $G$. 
The Dehn function of $\X(G)$ is bounded below by a cubic polynomial if $G$ maps onto a non-abelian free group.
\end{abstract}

\maketitle

\section*{Introduction}

In \cite{Said} Sa\"{i}d Sidki defined {{ } the weak commutativity} functor $\X$ {{ } which} assigns to a group $G$ the group
$\X(G)$  obtained from the free product $G\ast G$ by forcing
each element $g$ in the first free factor to commute with the copy of $g$ in the second free factor.
More precisely, taking a second copy $\-{G}$ of $G$ and fixing an isomorphism  $g\mapsto\-{g}$, one
defines $\X(G)$ to be the quotient of the free product $G\ast\-{G}$
by the normal subgroup $\<\!\langle   [g, \-g] :  g \in G \rangle\!\>$. 
At first glance, it seems that if $G$ is infinite then one is likely to require infinitely many relations
to define $\X(G)$, even if $G$ is finitely presented. 
But in \cite{BK1} we proved
that if $G$ is finitely presented then so too is $\X(G)$.  
Thus Sidki's construction
provides an intriguing new source of finitely presented groups. 

Our purpose in this article is two-fold: we investigate the complexity
of the word problem for $\X(G)$, with emphasis on Dehn functions and isoperimetric inequalities, and we make significant additions to the list of properties
that the functor $\X$ is known to preserve.  
 
Previous investigations have already shown
that the functor $\X$ preserves certain interesting classes of groups.
Sidki himself
\cite[Thm.~C]{Said} showed that if $G$ lies in
any of the following classes then $\X(G)$ lies in the same class: finite $\pi$-groups, where $\pi$ is a set of primes; finite nilpotent groups; solvable groups.
Gupta, Rocco and Sidki  \cite{G-R-S} proved that the class of finitely generated nilpotent groups is closed
under $\X$, and 
Lima and Oliveira \cite{LO} proved the same for polycyclic-by-finite groups.  
In \cite{KochSidki} Kochloukova and Sidki proved that if $G$ is a soluble group of type $\FP_{\infty}$ then $\X(G)$ is a soluble group of type $\FP_{\infty}$, but
in \cite{BK1} we proved  that when $F$ is a non-abelian free group,
$\X(F)$ is not of type $\FP_{\infty}$. 

The emergence of unexpected behaviour in $\X(G)$ is restricted almost entirely to 
groups with infinite abelianisation, as the following pair of contrasting results illustrates.

\begin{propA}
	If $G$ is a word-hyperbolic group that is perfect, then $\X(G)$ is biautomatic.
	In particular, $\X(G)$ is of type ${\rm{FP}}_\infty$ and satisfies a quadratic isoperimetric inequality.
\end{propA}

\begin{thmA}
	Let $G$ be a finitely presented group. If $G$ maps onto a non-abelian free group, then
	the Dehn function of $\X(G)$ is bounded below by a cubic polynomial, and 
	$\X(G)$ has a subgroup of
	finite index $X_0<\X(G)$ with $H_3(X_0,\Q)$ infinite dimensional.
\end{thmA}

There are many hyperbolic groups $G$ that are perfect but have
subgroups of finite index $G_0<G$ such that $G_0$ maps onto {{ } a non-abelian free group}: hyperbolic 3-manifold groups
that are integer homology spheres have this property, as do free products of finite
perfect groups, and the small-cancellation
groups obtained by applying a suitable version of the Rips construction \cite{rips} 
to perfect groups.
Thus Proposition A and Theorem B
underscore how radically $\X(G)$ can change when one replaces $G$ by a subgroup of
finite index.

The remainder of our results rely on a
structural result concerning the subgroups $L$ and $W {{ }  = W(G)}<\X(G)$ that we shall now describe. 
There are natural
surjections $\X(G)\to G\times \-{G}$ and $\X(G)\to G$, the latter defined by sending
both $g$ and $\-{g}$ to $g$. The kernels of these maps, $D$ and $L$, commute,
and the second map splits to give $\X(G)=L\rtimes G$.
Sidki \cite{Said} identified the crucial role that the abelian group $W:=D\cap L$ and the exact sequence
$$
1\to W \to \X(G) \overset{\rho}\to G\times G\times G
$$
play in understanding $\X(G)$, where $\rho(g)=(g,g,1)$ and $\rho(\-{g})=(1,g,g)$. 
{{ } (We sometimes write $W(G)$ instead of $W$.)}
  If $G$ is finitely presented and perfect, then $W<\X(G)$ is finitely generated and
central. But for finitely presented groups that are not perfect, $W$ need not be finitely
generated and the conjugation action of $\X(G)$ on $W$ can be difficult to understand.
As $DL<\X(G)$ acts trivially, the action of $\X(G)$ factors through $Q:=\X(G)/DL= G/G'$. 
(Throughout, we use the standard notation $G'=[G,G]$.) 
Thus $W$ is a $\Z Q$-module. 

The action of $\X(G)$ on $L$ by conjugation induces an action of $G$ on 
$L/L'$ that factors through $Q$, and our  analysis of $W$
as a $\Z Q$-module begins with the observation that this action is {\em{nilpotent}},
i.e. there exists an integer $d$ such that $[\l, g_1,\dots,g_d] \in L'$
for all $l\in L$ and all $g_i\in G$ (see Proposition \ref{l:nilp}).

We then consider the
$\mathbb{Z} Q$-module $M=((G'/G'') \otimes_{\mathbb{Z}} (G'/ G''))_{Q_0}$, where
the action of $Q$ on the first factor  
is induced by  conjugation in $G$ while the action on the
second factor is trivial, and the co-invariants in the definition
of $M$ are taken with respect to 
the action of $Q_0 = \{ (q, q^{-1}) \mid q \in Q \} \leq (G/G') \times (G/G') $
by conjugation on $(G'/G'') \otimes_{\mathbb{Z}} (G'/ G'')$. The apparent asymmetry in the $Q$-action is an illusion:
because we have factored out the action of $Q_0$,  the action of $Q=G/G'$ on $M$ can equally be described 
as being trivial  on the first factor of the tensor product and by conjugation on the second factor.

The following theorem explains why 
$M$ plays an important role in the understanding of $W$.  
Its proof relies heavily on homological methods.
\begin{thmA}\label{propB}
	Let $G$ be a finitely generated group, let $Q$ be the abelianisation of $G$,
	and consider $W$ as a $\mathbb{Z} Q$-module in the manner described above.
	Then there exists a submodule $W_0<W$ such that:
	\begin{enumerate} 
		\item $W_0$ is a subquotient of  the $\mathbb{Z} Q$-module
		$M=((G'/G'') \otimes_{\mathbb{Z}} (G'/ G''))_{Q_0}$; 
		\item the action of $Q$ on $W / W_0$ is nilpotent.
	\end{enumerate}
\end{thmA}

As a first illustration of the utility of these structural 
results (Proposition \ref{l:nilp} and Theorem \ref{propB}) we add to the list of properties that $\X(G)$ is known to 
preserve.  
We remind the reader that $\X(G)$ can change radically when one
replaces $G$ by a subgroup of finite index.

\begin{thmA} \label{t:nilp}
	If $G$ is finitely generated and virtually nilpotent, then $\X(G)$ is virtually nilpotent.
\end{thmA}

The virtual nilpotency class of $\X(G)$ will in general be greater than that of 
$G$. This is already the case for virtually abelian groups \cite{G-R-S}. 

Recall that a group $G$ is {\em{$n$-Engel}} if $[a,b,\ldots, b] = 1$
for all $a,b\in G$,  where $b$ appears $n$ times in the left-normed commutator.
We shall deduce 
the following result from Theorem \ref{propB}.
Our proof relies on Gruenberg's result that the
maximal metabelian quotient $G/G''$ of a finitely generated Engel
group is nilpotent \cite{gruen}.

\begin{thmA} If $G$ is a finitely generated $n$-Engel group, then $\X(G)$ is $m$-Engel 
	for $m = n+d+s+3$, where $d$ is the nilpotency class of $G/ G''$ and $s$ is the nilpotency class of the action of $G$ on $L/ L'$.
\end{thmA}

We return to results about Dehn functions, reminding the reader
that the Dehn function of a 
finitely presented group $G=\<A\mid R\>$ is the least function $\f_G(n)$ such that
any word $w$ in the kernel of ${\rm{Free}}(A)\to G$, with length $|w|\le n$,
can be expressed as a product of at most $\f_G(n)$ conjugates of the 
relations $r\in R$ and their inverses. $G$ is said to satisfy a 
polynomial isoperimetric inequality if $\f_G(n)$ is bounded above by a 
polynomial function of $n$. See Section \ref{s:dehn} for more details,
including the sense in which $\f_G(n)$ is independent of the chosen presentation. A finitely presented group has a soluble word problem if and only if  its
Dehn function is recursive.

\begin{thmA}\label{t:solvable}
Let $G$ be a finitely presented group. The word problem in $\X(G)$ is soluble if and only
if the word problem in $G$ is soluble.
\end{thmA}

Estimating the complexity of the word 
problem in $\X(G)$ by means of isoperimetric inequalities 
is a more subtle task.  
The proof of the following result relies heavily on the understanding of $W$
as a $\Z Q$-module that is established through Theorem \ref{propB}, as well
as results concerning the Dehn functions of subdirect products of groups \cite{dison}.

\begin{thmA}\label{t:dehn} 
	If $G$ is a finitely presented group whose maximal metabelian quotient $G/G''$ is
	virtually nilpotent, then there is a polynomial $p(x)$ such that 
	$$ \f_G(n) \preceq  \f_{\X(G)}(n) \preceq p\circ \f_G(n);$$
	in particular, $G$ satisfies a polynomial isoperimetric inequality if and only if
	$\X(G)$ satisfies a polynomial isoperimetric inequality (of different degree, in  (
	general).
\end{thmA}
{{ } 
The degree of the polynomial $p(x)$ depends on the torsion-free rank of $W$ (which is finite since $G'/G''$ is finitely generated  \cite{KochSidki}) and on the torsion-free rank of $G/ [G,G]$;
it can be estimated by following the proof of Theorem \ref{t:56}.}

Theorem \ref{t:dehn} applies to many groups of geometric interest. For example,
Delzant \cite{delz} proved that for any compact K\"ahler manifold $M$ that does not
fibre (equivalently, $\pi_1M$ does not map onto a surface group), the maximal
metabelian quotient of $\pi_1 M$ is virtually nilpotent. The maximal
metabelian quotient of the Torelli subgroup 
of the mapping class group of a surface of high genus is also virtually nilpotent, as is the corresponding subgroup of the outer automorphism group of
a free group \cite{ershov}.

\begin{corA}
	Let $G$ be the fundamental group of a compact K\"ahler manifold that does not fibre. If
	$G$ satisfies a polynomial isoperimetric inequality, then so does $\X(G)$.
\end{corA}

Theorem \ref{t:dehn} is already interesting for groups whose commutator subgroup is perfect.
A much-studied group  with this property is Thompson's group  
\begin{equation} \label{Fpresentation}
F= \langle x_0, x_1, \ldots, x_n, \ldots \mid x_i^{x_j} = x_{i+1} \hbox{ for } j < i \rangle. \end{equation}
It is easy to show that $F$ can be finitely presented. Famously, $F'$ is simple.
Brown and Geoghegan \cite{B-G} proved that $F$ is of type ${\rm{FP}}_\infty$, and 
Guba \cite{Guba} proved that  the Dehn function of $F$ is quadratic, building on 
his earlier work with Sapir \cite{GubaSapir}. 

\begin{corA}
	The Sidki double $\X(F)$ of Thompson's group $F$  satisfies a polynomial isoperimetric inequality. 
\end{corA}

{{ }It is also possible to estimate the Dehn function of $\X(G)$ is some situations where Theorem \ref{t:dehn} does not apply but where one nevertheless has control of $W$. This is the case, for example, 
when $G$ is a metabelian Baumslag-Solitar group ${\rm{BS}}(1,m)$ with $m>1$; Kochloukova and Sidki \cite{KochSidki} prove that $W = 1$ in this case, so it follows from Theorem \ref{Disonthm} that $\X(G)$ has an exponential Dehn function, since 
$G$ does.}

Related to $\X(G)$ one has the group $ {\mathcal{E}}(G) = ( G, \overline{G} \mid [\mathcal{D}, \mathcal{L}]=1)$  studied in \cite{LS}, where
$\mathcal{D} = [G,\-G]$ and ${\mathcal L} =  \langle \{ g^{-1} \-g \mid g \in G \} \rangle$ {{ } are subgroups of $G * \-G$}. The circumstances
under which  $ {\mathcal{E}}(G)$ is finitely presented are much more restrictive than for $\X(G)$; it is necessary
but not sufficient that $G$ be finitely presented --  see \cite{desi-EG}. Correspondingly, the
proof of the following result is much more straightforward than that of Theorem \ref{t:dehn}.

\begin{propA}\label{p:EG}
	Whenever $ {\mathcal{E}}(G)$ is finitely presented,  $\delta_{{\mathcal{E}}(G)}(n)
	\preceq \max\{n^4, \,  \delta_G(n)^2\}$.
\end{propA}

In the final section of this paper, we return to the search for classes of groups that are
closed under passage from $G$ to  $\X(G)$. Our focus now is on the growth of groups.
We remind the reader that the growth function of a group $G$ with  finite
generating set $S$ counts the number of elements in balls of increasing radius in the word metric $d_S$.
The growth type is polynomial, exponential, or intermediate (i.e.~sub-exponential but faster than any polynomial).

We shall see that the main arguments in the proof of Theorem \ref{t:nilp}
can be abstracted to provide the following criterion for closure under $\X$; the key point is the control on
the $\Z Q$-module  $W$ provided by Theorem \ref{propB}.

\begin{thmA} \label{grw} Let $\mathcal P$ be a class of finitely generated groups such that
\begin{enumerate}
\item[$\bullet$] every metabelian group in $\mathcal{P}$ is virtually nilpotent;
\item[$\bullet$] $\mathcal P$ is closed under the taking of finitely generated
subgroups,  quotients, central extensions,  extensions by and of finite groups,
and finite direct products.
\end{enumerate}
Then $\X(G) \in {\mathcal P}$ if and only if $G \in {\mathcal P}$.
\end{thmA}
If one drops the requirement that $\mathcal P$ is closed under direct products, then the conclusion
has to be modified:  $\X(G) \in {\mathcal P}$ if and only if ${\rm{im }} (\rho) \in {\mathcal P}$,
where $\rho:\X(G)\to G\times G\times G$ is as described above.
Both versions of the theorem remain true if the only central extensions under which $\mathcal{P}$
is assumed closed are those with finitely generated kernel. And in either version, it follows
from \cite{KochSidki} that $W(G)$ is finitely generated as an abelian group for all $G\in{\mathcal P}$.

Gromov \cite{Gromov} proved that groups of polynomial growth are virtually nilpotent, and the converse is
straightforward. Thus the first part of the following corollary is covered by Theorem \ref{t:nilp}. 
The closure properties required in the other cases
are discussed in Section \ref{s:growth}.

\begin{corA} \label{grw2}
Let $G$ be a finitely generated group. Then
\begin{enumerate}
\item  $\X(G)$ has  polynomial growth if and only if $G$ has polynomial growth;
\item $\X(G)$ has exponential growth if and only if $G$ has exponential growth;
\item $\X(G)$ has intermediate growth if and only if $G$ has intermediate growth.
\end{enumerate}
\end{corA}

In the context of Theorem \ref{grw} we shall also discuss  the class of finitely generated
groups that have the fixed-point property for cones, as studied by Monod in \cite{Monod} (see Section \ref{s:last}).

This paper is organised as follows. After gathering necessary preliminaries, in Section 2 we prove Proposition A and Theorem B,
and in Section 3 we prove Theorem D. In Section 4 we prove Theorem C and a related structural result that is needed
in later sections, Corollary \ref{action21}. Sections 5 and 8 contain our results concerning isoperimetric inqualities and word problems.
Engel groups are discussed in Section 6 and Proposition \ref{p:EG} is proved in Section 7. In the final section, we extract a proof of
Theorem \ref{grw} from the arguments in earlier sections, discuss the growth of groups, and prove Corollary \ref{grw2}.

\medskip
\noindent{\em Acknowledgement.} We thank Nicolas Monod for correspondence concerning Section \ref{s:last}. 

\section{Preliminaries}

Throughout, our notation for conjugation is $x^y:=y^{-1}xy$, our commutator convention
is $[x,y]=x^{-1}y^{-1}xy$, and iterated commutators are left-normed 
$[a_1,\dots,a_n]:=[[a_1,\dots,a_{n-1}],a_n]$. 

We assume that the reader is familiar with the use of homological methods in
group theory as described in \cite{bieri} and \cite{brown}.

\subsection{On the structure of the group $\X(G)$} \label{s:structure} 

We recall some subgroups and decompositions of $\X(G)$ defined by Sidki \cite{Said}.
We follow the notations from \cite{Said} except that we write $\-G$ and $\-g$
where Sidki writes $G^\psi$ and $g^\psi$. In the introduction we defined
$$
D = D(G) := [G,\-{G}]\ \ {\rm{and}}\ \  L = L(G) :=  \<\!\< \{ g^{-1} \-{g} \mid g \in G \} \>\!\>.
$$ 
Note that $D$ is the kernel of the natural map $\X(G)\to G\times\-{G}$ and that
$L$ is the kernel of the map $\X(G)\to G$ that sends both $g$ and $\-{g}$ to $g$.
This last map has an obvious splitting
$$
\X(G) = L\rtimes G.
$$

We adopt the notation $\l_g = g^{-1}\-{g}\in L$, where $g\in G$.
Following \cite{Said}, one sees that $L$ is actually generated (not just normally
generated) by the set of $\l_g$, 
because
\begin{equation}\label{Lnormal}
\l_u^x = x^{-1}u^{-1}\-{u}x = (ux)^{-1}\-{u}(\-x\-x^{-1})x = (ux)^{-1}(\-{ux})(\-x^{-1}x)= \l_{ux}\l_x^{-1}
\end{equation}
for all $x,u\in G$, and $\l_u^{\-x} = \l_u^{x\l_x} = \l_x^{-1}\l_u^x\l_x = \l_x^{-1}\l_{ux}$.

By taking the direct product of the maps with kernels  $L$ and $D$ (and re-ordering the factors) we obtain a map
$$
\rho: \X(G) \to G\times {G}  \times  \-{G}  \cong G\times G\times G$$
with 
$$
\rho(g) = (g,g,1), \ \ \rho(\-g) = (1,g,g) \hbox{ for all } g \in G.
$$
The kernel of $\rho$ is
$$\ W = W(G) := D \cap L.$$
Sidki showed in  \cite[Lemma~4.1.6~(ii)]{Said}  that $D$ commutes with
$L$ and therefore $W$  is central in $DL$; in particular, $W$ is abelian. 

The image of $\rho$ is
$$
\{ (g_1, g_2, g_3) \mid g_1 g_2^{-1}g_3 \in [{G}, G] \}.$$
Note that $\imm ( \rho)$ contains the commutator subgroup of $G\times G \times G$; {{ } it is the kernel of the epimorphism $G\times G\times G \to G/[G,G]$ whose restriction to the first and third cooridinates is $g\mapsto g[G,G]$ and whose restriction to the
second coordinate is $g\mapsto g^{-1}[G,G]$. }
The projection of  $\imm ( \rho)$ to each pair of coordinates in $G\times G\times G$ is onto, so by the Virtual Surjection to Pairs Theorem \cite{BHMS}, if
$G$ is finitely presented then $\imm ( \rho)$ is finitely presented. 

For the remainder of the paper, we fix the notation $Q:=G/[G,G]$. The action of $G<\X(G)$
by conjugation on $W$ and on $L/L'$ (hence on the homology of $L/W$) factors through $Q$.

\begin{lemma}  \label{nilpotent12} As $\Z Q$-modules, $ W \cap L'$ is a quotient of
	$H_2(L/ W, \mathbb{Z})$.
\end{lemma}

\begin{proof} The 5-term exact sequence in homology associated to the short exact
	sequence $1\to W \to L \to L/W\to 1$ (where $W$ is central) gives an exact sequence of $\Z Q$-modules
	$$
	H_2(L/W, \Z) \to H_0(L/W, W) \to H_1(L,\Z) \to H_1(L/W,\Z)\to 0
	$$
	and the kernel of $H_0(L/W, W)=W\to H_1(L,\Z)$ is $W \cap L'$. 
\end{proof} 

We shall need later the following result of Kochloukova and Sidki \cite{KochSidki}.

\begin{lemma}\label{l:KS} Suppose that
	 $G$ is of type ${\rm{FP}}_2$. If $G'/G''$ is finitely generated, then so is   $W(G)$.
\end{lemma}

\subsection{Dehn functions and isoperimetric inequalities}\label{s:dehn}

We recall some standard facts and terminology concerning the geometry of the word problem
in finitely presented groups. The reader unfamiliar with this material may wish to consult
\cite{bfs} for a more thorough introduction. 

Let $\P \equiv \< X \mid R\>$ be a finite presentation of a group $G$. By definition,
a word $w$ in the free group $F(X)$ represents the identity in $G$ if and only if 
it lies in the normal subgroup generated by the set $R$, which means that there is
an equality in $F(X)$ of the form
\begin{equation}\label{w:equ}
w = \prod_{i=1}^M \theta_i^{-1} r_i \theta_i
\end{equation}
with $r_i\in R\cup R^{-1}$ and $\theta_i\in F(X)$. The difficulty
of solving the word problem in $G$ by means of a na\"{i}ve attack seeking equalities of
this form depends on the size of $M$ in a minimal such expression, as well as the length
$|\theta_i|$ of the conjugating elements $\theta_i$. 

Van Kampen's Lemma establishes a correspondence between equalities in the free group
of this form (more formally, sequences $(\theta_i, r_i)_{i=1}^M$) and a class of contractible, labelled, planar diagrams; see, e.g. \cite{bfs}.
This correspondence enables one to invoke geometric arguments in pursuit of understanding 
the complexity of the word problem in $G$. It also 
explains the  following  terminology: $M$, the number of factors in the
product on the righthand side of the above equality, is defined to be 
the {\em{area}} of the product  and $\max \{ |\theta_i| : 1\le i \le m\}$ is defined
to be its {\em{radius}}.

For each word $w\in F(X)$ that represents the identity in $G$, one defines
${\rm{Area}}(w)$ to be the least area of any expression of this form for $w$, and
the {\em{Dehn function}} of $\P$ is the function $\delta: \N \to \N$ defined by
$$
\delta(n) = \max \{{\rm{Area}}(w) \mid w=_G1 \ {\rm{and}} \ |w|\le n\}.
$$

A pair of functions $(\alpha, \rho)$, with $\alpha, \rho : \N \to \N$, 
is called an {\em{area-radius pair}} for $P$
if for each word $w\in F(X)$ that represents the identity in $G$,
if $|w|\le n$ then there is an equality of the form (\ref{w:equ})
with area at most $\alpha(n)$ and radius at most $\rho(n)$. (It is important
to note that these bounds hold simultaneously.)

Different finite presentations of a group $G$
will give rise to different Dehn functions,
but they will be equivalent in the following sense. For functions $f,g:\N\to\N$,
define  $f \preceq g$ if there exists a 
positive integer  $C$ such that
$$
f(n) \leq C g(Cn + C) + Cn + C
$$
for all $n \in \N$, and define $f\simeq g$ if  $f \preceq g$ and $g  \preceq f$.

Likewise, if $(\alpha, \rho)$ is an area-radius pair for a finite presentation of
$G$ and $\P'$ is any other finite presentation of $G$, then there will exist
functions $\alpha', \rho' :\N\to \N$ with $\alpha\simeq \alpha'$ and $\rho \simeq
\rho'$ such that $(\alpha', \rho')$ is an area-radius pair for $\P'$ (see e.g.
\cite[Prop 3.15]{dison}).

Following common practice, we shall write $\delta_G$ to denote ``the" Dehn function of
a finitely presented group, with the understanding that this is only well-defined up
to $\simeq$ equivalence. Likewise, we shall refer to $(\alpha, \rho)$ as an area-radius pair for $G$ if $\alpha$ and $\rho$ are $\simeq$ equivalent to the functions in an
area-radius pair for a finite presentation of $G$.

It is easy to see that for constants $p,q\ge 1$, the functions $n\mapsto n^p$
and $n\mapsto n^q$ are $\simeq$ equivalent if and only if $p=q$.

A function $\beta : \N\to \N$ is called an {\em{isoperimetric function}} for $G$ if
$\delta_G \preceq \beta$. One says that $G$ {\em{satisfies a polynomial isoperimetric
		inequality}} if there is a constant $d$ such that $\delta_G(n) \preceq n^d$.
And an area-radius pair $(\alpha, \rho)$ for $G$ is said to be polynomial if 
both $\alpha$ and $\rho$ are bounded above by polynomial functions.

A simple cancellation argument (or diagrammatic argument) establishes the following
(see e.g. \cite[Lemma 2.2]{gersten} or \cite{bfs}). 

\begin{lemma} \label{l:pair} Let $G$ be a finitely presented
	group. If $\beta$ is an isoperimetric function for $G$, then (up to $\simeq$ equivalence)
	$(\beta, \beta)$ is an area-radius pair for $G$.
\end{lemma} 

\begin{corollary}
	If a finitely presented group satisfies a polynomial isoperimetric inequality, then
	it has a polynomial area-radius pair.
\end{corollary}

We need two other easy and well-known facts about isoperimetric functions, the second of which
is a special case of the quasi-isometric invariance of Dehn functions.

\begin{lemma} \label{l:retract} If $H$ is a retract of the finitely presented group $G$, then  (up to $\simeq$ equivalence)
	every area-radius pair for $G$ is also an area-radius pair for $H$.
\end{lemma} 

\begin{lemma} \label{l:qi} Let $\phi:H \to G$ be a homomorphism of finitely presented groups.
	If the kernel  of $\phi$ is finite and $\phi(H)$ has finite index in $G$,
	then the Dehn function of $H$ is $\simeq$ equivalent to the Dehn function of $G$. 
\end{lemma} 

\subsection{Biautomatic groups}

The theory of automatic groups grew out of investigations into the algorithmic 
structure of Kleinian groups by Cannon and Thurston, and it was developed thoroughly
in the book by Epstein {\em{et al.}} \cite{epstein}. Let
$G$ be a group with finite generating set $A=A^{-1}$ and let $A^*$ be the free monoid
on $A$ (i.e. the set of all finite words in the alphabet $A$). An
{\em{automatic structure}} for $G$ is determined by a normal
form $\mathcal{A_G}=\{\sigma_g \mid g\in G\}\subseteq A^*$ such that $\sigma_g=g$ in $G$. 
This normal form is required to satisfy two conditions: first, $\mathcal{A_G}\subset A^*$
must be a {\em{regular language}}, i.e. the accepted language of a finite state automaton;
and second, the edge-paths in the Cayley graph $\mathcal{C}(G,A)$ that begin at
$1\in G$ and are labelled by the words $\sigma_g$ must satisfy the following {\em{fellow-traveller condition}}: there is a constant $K\ge 0$
such that for all $g,h\in G$ and all integers $t\le \max\{|\sigma_g|, |\sigma_h|\}$,
$$
d_A(\sigma_g(t), \sigma_h(t))\le K \, d_A(g,h),
$$
where $d_A$ is the path metric on  $\mathcal{C}(G,A)$ in which each edge has length $1$,
and $\sigma_g(t)$ is the image in $G$ of the initial subword of length $t$ in $\sigma_g$.

A group is said to be {\em{automatic}} if it admits an automatic structure. If $G$
admits an automatic structure with the additional property that for  all integers $t\le \max\{|\sigma_g|, |\sigma_h|\}$,
$$
d_A(a.\sigma_g(t), \sigma_h(t))\le K \, d_A(ag,h),
$$
for all $g,h\in G$ and $a\in A$, then $G$ is said to be {\em{biautomatic}}. 
Biautomatic
groups were first studied by Gersten and Short \cite{GS}. Automatic and biautomatic groups form two of the most
important classes  studied in connection with notions of non-positive curvature in group theory; see \cite{mrb:camb} for a recent survey.  From the point of view of
this article, the most salient properties of automatic groups are the following.

\begin{prop}
	If a finitely generated group $G$ is automatic, then it is finitely presented and
	of type ${\rm{FP}}_\infty$. Moreover,  $(\alpha(n),\rho(n)) = (n^2, n)$ 
	is an area-radius pair for $G$.
\end{prop}

Biautomatic groups  behave better than automatic groups with respect to central quotients \cite{mosher} and this is needed in our proof of Proposition A.

\section{Proof of Proposition A and Theorem B} 

\subsection{Proof of Proposition A}

We remind the reader that every perfect group $\G$ admits a universal central extension $1\to H_2(\G,\Z)\to \tilde{\G}\overset{p}\to \G\to 1$, meaning that
for all perfect groups $E$ and surjections $\pi:E\to \G$ with central kernel,  there is a surjection  
(central quotient) $s:\tilde\G\to E$ with $\pi\circ s =p$. 

\begin{lemma}\label{l:perfect}
	If $G$ is perfect, then $1\to W \to \X(G)\overset{\rho}\to G\times G \times G \to 1$ is exact, $W$ is central in $\X(G)$, and
	$\X(G)$ is a central quotient of the universal central extension of $G\times G \times G$. 
\end{lemma}

\begin{proof} {{ } Let $\l_x=x^{-1}\-x$. Note that $\rho (\l_a \l_b \l_{ab}^{-1}) = ([a,b], 1,1)$, from which it follows that} the image of $DL$ under $\rho$ contains the commutator subgroup  of $P:=G\times G \times G$, which is the whole of $P$, and
	$W$ is central in $DL$. Thus $W$ is central. And since  $\X(G)$ is perfect, the last assertion is an instance of the universal property recalled
	above.
\end{proof} 

\noindent{\em{Proof of Proposition A:}} Assume now that the perfect group $G$ is hyperbolic. Neumann and Reeves \cite{NR} proved that every central extension of a hyperbolic group with finitely generated kernel is
biautomatic, so in particular this is true of the universal central extension $\tilde G$. Let $P=G\times G \times G$.
By the  K\"unneth formula $H_2(P,\Z)= H_2(G,\Z)\times H_2(G,\Z)\times H_2(G,\Z)$, so the universal central
extension of $P$ is $\tilde{G}\times \tilde {G} \times \tilde{G}$. A direct product of biautomatic groups is biautomatic 
and Mosher \cite{mosher} proved
that a central quotient of a biautomatic group is again biautomatic, so Lemma \ref{l:perfect} completes the proof of Proposition A. 


\begin{remark}  The preceding argument establishes the following more general fact:
	let $\mathcal{C}$ be any class of finitely presented groups that is closed under the formation of finite direct sums, central extensions by finitely generated abelian groups, and central quotients; if a finitely presented perfect group $G$ belongs to $\mathcal{C}$, then so does $\X(G)$.
\end{remark}

\subsection{Proof of Theorem B}

Throughout this section, $F$ will denote the free group of rank $2$ with basis $\{a,b\}$. Theorem B reduces easily to the special
case $G=F$: one can split the surjection $G\to F$ to regard $F$ as a retract of $G$; any retraction $A\to B$ extends to a retraction
$\X(A)\to \X(B)$; and if $C\to D$ is a retraction of finitely presented groups then it is easy to see that the Dehn function of $C$ is bounded
below by that of $D$. Thus we concentrate on showing that the Dehn function of $\X(F)$ is bounded below by a cubic polynomial.

\begin{definition}\label{d:dist} Let $H<G$ be a pair of finitely generated groups and let $d_H$ and $d_G$ be the word metrics associated to
	a choice of finite generating sets. The {\em distortion} of $H$ in $G$ is the function
	$$
	\dist_H^G(n) = \max \{ d_H(1,h) \mid h\in H \ {\rm{ with }}\ d_G(1,h)\le n\}.
	$$
	One checks easily that, up to Lipschitz equivalence, this function does not depend on the choice of word metrics.
\end{definition}

\begin{lemma}\label{l:chain} Let $A<B<C$ be finitely generated groups. If $B$ is a retract of $C$, then  $\dist_A^B(n)\simeq \dist_A^C(n)$.
\end{lemma}

\begin{proof}
	If $r:C\to B$ is a retraction, then we obtain a finite generating set $S$ for $C$ by taking a finite generating set $S'$ for $B$ and adding finitely
	many generators from the kernel of $r$. With respect to the associated word metrics, $d_B(1,b)= d_{C}(1,b)$ for all $b\in B$, {{ } in particular for $b = a \in A$.}
\end{proof}

We need the following connection between Dehn functions and distortion. We remind the reader that the {\em trivial HNN extension}
$B\dot{\ast}_H$
of a group $B$ with associated subgroup $H<B$ is the quotient of $B\ast \<t\>$ by the normal subgroup generated by $\{[t,h]\mid h\in H\}$.

The following lemma is a slight variant on an inequality in  Theorem III.$\Gamma$.6.20 of \cite{BH} (see also \cite{AO}).

\begin{lemma}\label{l:distort} Let $H<B<G_0<G$ be groups with $G_0 = B\dot{\ast}_H$ a trivial HNN extension. If $G$ is finitely
	presented and $G_0<G$ has finite index, then the Dehn function of $G$ satisfies
	$$ n\ \dist_H^{G}(n)  \preceq \f_G(n) .$$
\end{lemma}

\begin{proof} Dehn functions and distortion are unaffected (up to $\simeq$ equivalence) by passage to subgroups of finite index, so we may
	assume that $G=G_0$. Killing the stable letter of the HNN extension retracts $G=G_0$ onto $B$,  so $ \dist_H^{G}(n) \simeq \dist_H^{B}(n)$,
	by Lemma \ref{l:chain}. Note that since $G$ is finitely presented, $B$ is finitely
	presented and $H$ is finitely generated.
	We fix a finite presentation $B=\< X \mid R\>$ where $X$ contains a finite generating set $Y$ for $H$. Then
	$G= \< X, t \mid R, \, [t,y]\, (y\in Y)\>$. For each positive integer $n$, we choose a word $w_n$ in the free group on $X$ that represents an element  $h_n\in H$ with $d_H(1,h_n)=\dist_H^B(n)$ and $d_B(1,h_n) \leq n$.  A standard argument using the $t$-corridors
	introduced in \cite{BridsonGersten} shows that any van Kampen diagram for the word $W_n = t^n w_n t^{-n} w_n^{-1}$ must contain
	at least $n\ \dist_H^B(n)$ 2-cells (see \cite{BH} pp.~506-508). And since $|W_n| \le 4n$, this establishes the desired lower bound on  $\f_G(n)$.
\end{proof}

Lemma \ref{l:distort} is relevant to our situation because of the following observation from \cite{BK1}.

\begin{lemma} \label{MarshallHall} There is a finite-index subgroup $G_0<\X(F)=L\rtimes F$ that contains $L$ as the associated subgroup
	of a trivial HNN extension $G_0= B\dot{\ast}_L$.
\end{lemma}

\begin{proof} Let $\pi : \X(F)=L\rtimes F \to F$ be the retraction with kernel $L$
	and recall that $[L,D] = 1$.
	For a fixed $t \in D \setminus L$ consider $\pi(t) \in F$. Marshall Hall's theorem \cite{Hall} provides a subgroup  $F_0<F$ of finite index such that $\pi(t)$ is a primitive element of $F_0$. Let $G_0 = \pi^{-1} (F_0) = L \rtimes F_0$. By choosing a 
	different section of $\pi$ if necessary, we may assume that $t\in \tilde F_0 = 1\times F_0$.
	We fix a free splitting $\tilde F_0 = C * F_1$, where $C$ is the cyclic subgroup  generated by $t$,
	and set  $B = L \rtimes \tilde F_1$. Then  $G_0= B\dot{\ast}_L$.
\end{proof}

In order to apply Lemma \ref{l:distort}, we need a lower
bound on  $\distLX(n)$. (The bound that we establish is sharp, but we don't require this
so omit the proof.)

\begin{lemma}\label{l:distLX} 
	The distortion of $L$ in $\X(F)$ satisfies $\distLX(n)\succeq n^2$.
\end{lemma}

\begin{proof} We maintain
	the notation $\l_x=x^{-1}\-x$ for elements of $L$, where $x\in F$.
	We proved in \cite{BK1} that $L$ is generated by  $\{\l_a, \l_b, \l_{ab}\}$. In the current setting, it is convenient to replace
	$\l_{ab}$ by $\lambda = \l_a\l_b\l_{ab}^{-1}$. Thus we work with the finite generating set $\{\l_a, \l_b, \lambda\}$ for $L$
	and $\{\l_a, \l_b, \lambda, a, b\}$ for $\X(F)=L\rtimes F$.
	
	Consider the elements $c_n:= \l_{a^n}\l_{b^n}\l_{a^nb^n}^{-1}\in L$. 
	To see why we focus on these elements,
	note that $\rho(c_n)=([a^n,b^n],1,1)\in F\times F\times F$.
	
	$c_n$ can be expressed as a word of length $6n$
	in the generators $\{ a,b, \l_a, \l_b\}$, because $\l_{a^n}=\l_a^n$ and $
	\l_{b^n}=\l_b^n$, while by (\ref{Lnormal}),
	\begin{equation}\label{cn}
	\l_{a^nb^n} = \l_{a^n}^{b^n}\l_{b^n} =  b^{-n} \l_a^{n} b^n \l_b^n.
	\end{equation} 
	We claim that $d_L(1,c_n) \geq n^2$. To prove this, we analyse the structure of 
	words $w$ in the free group $\mathcal F$ with basis $\{\l_a, \l_b, \lambda\}$
	that equal $c_n$ in $L$. As always, we write $u^v$ as shorthand for $v^{-1}uv$.
	By making repeated use of the identity $vu^v = uv$, we can express $w$ in $\mathcal F$
	as a product of the form
	\begin{equation}\label{W}
	w = V\prod_{i=1}^N \lambda^{\pm\theta_i} 
	\end{equation}
	with $V=V(\l_a,\l_b)$ and $\theta_i$ words in the letters 
	$\l_a^{\pm 1},\l_b^{\pm 1}$, and $N\le |w|$.
	Consider the composition $\overline p$ 
	of $\mathcal{F}\to L$ with the map $\-\pi:\X(F)\to \-F$ that kills $F$ and
	is the identity on $\-F$.  Note that $\-p(\l_a) = \-a, \-p(\l_b)=\-b$ and
	$\-p(w) = \-\pi(c_n) =1$. By taking the image of the free equality (\ref{W})
	under $\-p$, we conclude that $V(\-a,\-b)=1$ in $\-F$. Hence $V$ is the
	empty word and we have an equality 
	$$
	w = \prod_{i=1}^N \lambda^{\pm\theta_i}  
	$$
	in $\mathcal{F}$. 
	
	Next we consider the composition $p$ 
	of $\mathcal{F}\to L$ with the map $\pi:\X(F)\to F$ that kills $\-F$ and
	is the identity on $F$.  
	This map sends $\lambda$ to $[a,b]$ and $w$ to $[a^n,b^n]$, while sending
	$\theta_i(\l_a,\l_b)$ to $\theta_i'=\theta_i(a^{-1},b^{-1})$. Thus we obtain an equality
	$$
	[a^n,b^n] = \prod_{i=1}^N  [a,b]^{\pm\theta_i'} 
	$$
	in the free group $F=F(a,b)$. A standard exercise (often used to motivate 
	van Kampen's lemma) shows that $[a^n,b^n]$ cannot be expressed in $F$ as a product of
	fewer than $n^2$ conjugates of $[a,b]$. Thus $n^2\le N\le d_L(1,c_n)$, as claimed.
\end{proof}

\noindent{\bf{Proof of Theorem B.}} We have a retraction $r:\X(G)\to \X(F)$.
In \cite{BK1} we proved that $\X(F)$ has a subgroup $\G$ of finite index with
$H_3(\G,\Q)$ infinite dimensional, and this injects into $H_3(r^{-1}(\G),\Q)$.

The Dehn function of $\X(G)$ is bounded below by the Dehn function
of its retract $\X(F)$.
From Lemmas \ref{l:distort} and \ref{MarshallHall} we have
$\f_{\X(F)}(n) \succeq n\ \dist_L^{\X(F)} (n)$. And in Lemma \ref{l:distLX}
we proved that $n^2\preceq\distLX (n)$.

\section{$\X$ preserves virtual nilpotence} 

Throughout this section we assume that $G$ is finitely generated.
We shall prove that if $G$ is virtually nilpotent then $\X(G)$ is virtually nilpotent. In \cite{G-R-S}, Gupta, Rocco and Sidki used commutator calculations to prove that if $G$
is nilpotent then so is $\X(G)$, and gave a bound on the nilpotency class. 
In this nilpotent case, our proof is shorter and more homological, but it does not give as  good a bound as theirs  on the nilpotency class of $\X(G)$.

\subsection{Nilpotent actions} 

If $A$ is an abelian group and  $B$ is a group acting on $A$ (on the right), then one can form the semidirect product $A\rtimes B$ in which the action $a^b$ is transformed into
conjugation. In multiplicative notation, $[a,b] = a^{-1}b^{-1}ab = a^{-1}a^b$. The
action is said to be {\em nilpotent} if there is an integer $d$ such that, for all
$a\in A$ and $b_1,\dots,b_d\in B$,
$$
[a,b_1,\dots,b_d]=1.
$$
If $B$ is nilpotent, then $A\rtimes B$ will be nilpotent if and only if the action of
$B$ on $A$ is nilpotent.

If we write the group operation in $A$ additively and regard $A$ as a $\Z B$ module,
writing the action $a^b$ as $a\circ b$, then $[a,b] = a\circ (b-1)$ and the vanishing
of the above commutator becomes 
$$
a\circ (b_1-1)\dots (b_d-1) =0.
$$ 
Thus $A$ is a nilpotent module over $\A(\Z B)$, the augmentation
ideal of $\Z B$. 

We retain the notation $L=L(G)$ for the normal subgroup of $\X(G)$
generated by the elements $\l_g= g^{-1}\-g$.
We proved in \cite{BK1} that when $G$ is finitely
generated, $L$ is finitely generated.  Lima and Oliveira \cite{LO} had proved earlier
that $L/L'$ is finitely generated. The action of $\X(G)=L\rtimes G$ by conjugation 
on $L/L'$ factors through $G$ (and even  $Q=G/G'$).

\begin{proposition}\label{l:nilp} For all finitely generated groups $G$,
	the action of $G$ on $L/ L'$ is nilpotent. 
\end{proposition}

\begin{proof} 
	Following \cite{Said}, we consider the set-map $L\to \A(\Z G)$
	that sends $\l_g$ to $(g-1)$. In order to make this a group homomorphism, we must 
	take a  quotient of $\A(\Z G)$ to force the image of $\l_{g^n}$ to coincide with that of $\l_g^n$, that is,
	$g^n-1 = n\, (g-1)$, the group operation in $\A(\Z G)$ being addition. A simple induction shows that it is enough to factor
	out by the $\Z G$-ideal $I_2=\<(g-1)^2 : g\in G\>$. From \cite{Said} and \cite{LO}
	we have that
	the resulting map $L/L'\to \A(\Z G)/I_2$ is an isomorphism of abelian groups.
	
	Moreover, with the action on the target coming from multiplication in 
	the ring $\Z G$, this map 
	is an isomorphism of right $\Z G$-modules, 
	because the image of   $\l_g^x =\l_{gx}\l_x^{-1}$
	is 
	$$(g-1)x = (gx-1) - (x-1).$$
	(In the light of this we omit $\circ$ from our notation.)
	
	Thus the task of showing that the action of $G$ on $L/L'$ is nilpotent is translated
	into showing that $V (\A(\Z G))^d=0$ for some $d\in\N$, where $V = \A(\Z G)/I_2
	\cong L/L'$;
	equivalently, that $(\A(\Z G))^{d+1} \subset I_2$.
	
	Moreover, because $G'=[G,G]$ acts trivially (via conjugation)  on $L/L'$, we can analyse  the action of the
	{\em commutative} quotient ring $\Z Q/I_0$ rather than $\Z G/I_2$,
	where $I_0$ is the ideal generated by $\{(q-1)^2\mid q\in Q\}$. Taking
	advantage of this commutativity, for $q_1, q_2 \in Q$ we have in $\Z Q/ I_0$
	$$
	2 q_1 q_2 - 1 = (q_1 q_2)^2 = q_1^2 q_2^2 = (2 q_1 - 1) (2 q_2 - 1) =
	4 q_1 q_2 - 2 q_1 - 2 q_2 + 1,
	$$
	hence
	$$
	2 (q_1-1) (q_2 - 1) = 2 q_1 q_2 - 2 q_1 - 2 q_2 + 2 = 0.
	$$
	It follows that
	{{ } $$2 (\A (\mathbb{Z} G))^2  \subseteq {{\rm{ker}}} (\mathbb{Z} G/ I_2 \to \mathbb{Z} Q/ I_0).$$}
	In other words, 
	\begin{equation}\label{nil1}
	2 V (\A (\mathbb{Z} G))^2 = 0. 
	\end{equation}
	
	Because $V\cong L/L'$ is finitely generated as an abelian group,  $\-{V}=V/2V$ is 
	finite, say with $k$ elements. The action of $(g-1)$ on $\-{V}$ by right multiplication
	depends only on the image of $(g-1)$ in $\-V$, so the action of any product
	$
	\prod_{i=1}^{k+1} (g_i-1) 
	$
	is the same as that of a product with a repeated factor. And since the action
	factors through the commutative ring $\Z Q$, the order of the factors does not matter.
	As  $(g_i-1)^2=0$ in $\Z G/I_2$, we conclude that $(\A(\Z G))^{k+1}$ annihilates $\-V$.
	In other words,
	\begin{equation} \label{nil2}
	V(g_1-1) \ldots (g_{k+1} - 1) \subseteq 2 V.
	\end{equation}
	By  (\ref{nil1}) and (\ref{nil2}),
	$$
	V (\A(\mathbb{Z} G))^{k+3} \subseteq 2V(\A(\mathbb{Z} G))^{2}  = 0.
	$$
	This completes the proof.
\end{proof}

\begin{remark}
	The proof given above shows that for every finitely generated group $G$,
	the semidirect product $(L/L')\rtimes Q$ contains a subgroup of finite index,
	namely $2(L/L')\rtimes Q$, that is nilpotent of class {{ } at most }$2$.  
\end{remark}

\subsection*{Proof of Theorem D}

Consider the sequence of subgroups
$$
W \cap L' \subseteq W \subseteq L \subseteq \X(G)
$$
and note that $W / (W \cap L') \cong W L' / L' \subseteq L/ L' \subseteq  \X(G)/ L'$. 
Let $V=L/L'$. From the decomposition $\X(G)=L\rtimes G$ we have
$\X(G)/L' \cong V\rtimes G$, and  Proposition \ref{l:nilp} tells us that the
action of $G$ on $V$ is nilpotent. Thus, if $G_0$ is a nilpotent subgroup of finite index in $G$ then  $V\rtimes G_0$ is nilpotent and  $\X(G)/L'$ is virtually nilpotent.

Let $T$ denote $L \rtimes G_0$, the preimage of $V \rtimes G_0$ in $\X(G)$; note that
$[\X(G) : T] < \infty$.

The nilpotent action of  $G$ on $V$ restricts to a nilpotent action on
$W/ (W \cap L') \cong WL' / L' $, so for some natural number $m$ we have   \begin{equation} \label{nov1} W  (\A ( \mathbb{Z} T))^m = [W, T, \ldots, T] \subseteq L' \cap W,\end{equation} 
where we view $W$ as a right $\mathbb{Z} \X(G)$-module via conjugation. Since $\X(G) / W \cong {\rm{im }} (\rho)$ is a subgroup of $G \times G \times G$, it is virtually nilpotent. Hence there is a normal subgroup $T_0<T$  that contains $W$ and is such that  $T_0/ W$ is nilpotent. 
Thus we have a short exact sequence of groups
$$
1\to W / (W \cap L') \to T_0 / (W \cap L') \to T_0 / W \to 1,
$$
where the nilpotent group $T_0 / W$ acts nilpotently on the abelian group $W / (W \cap L')$ by (\ref{nov1}). Hence
\begin{equation} \label{nilpotent11}
T_0 / (W \cap L') \hbox{ is nilpotent.}
\end{equation}
We will be done if we can show that $T_0$ is nilpotent.

As $T_0$ is normal, $[L/W, T_0/W]\subseteq T_0/W$. And as $T_0/W$ is nilpotent, for
a sufficiently long commutator $[L/W, T_0/W, \ldots , T_0/W] \subseteq [T_0/ W, \ldots, T_0/W] = 1$; in other words, $T_0/W$ acts  nilpotently on $L/W$. Therefore \begin{equation} \label{nilp-action} T_0/W \hbox{ acts nilpotently on  }H_2(L/W, \mathbb{Z}).
\end{equation} 

The result now follows from Lemma \ref{nilpotent12}, but we give more details.
The central extension
$$ 1 \to W \cap L' \to L \to L/ (W \cap L') \to 1$$ is stem  (i.e. the 
central subgroup $W \cap L'$
is contained in the commutator of the middle group $L$), so by the general
theory of central extensions, 
\begin{equation} \label{nilp12} W \cap L' \hbox{  is a quotient of }H_2(L/W, \mathbb{Z}).\end{equation} 
The quotient map, which is given by Hopf's formula, is equivariant with
respect to the action of $T_0/W$, so from
(\ref{nilp-action}) and (\ref{nilp12}) we have that $T_0/W$ acts nilpotently on $W \cap L'$. Thus 
\begin{equation} \label{nilpotent13} \hbox{ the action by conjugation of } T_0 \hbox{ on }
W \cap L' \hbox{ is nilpotent }.
\end{equation}  Together, (\ref{nilpotent11}) and (\ref{nilpotent13}) imply that   $T_0$ is nilpotent.

\section{The structure of $W(G)$ as a $\Z Q$-module}\label{s:structures}

In this section we prove Theorem \ref{propB}, which describes the structure of 
$W=W(G)$ as a $Q$-module, where the action of $Q = G/G'$ on $W$ is induced by 
the action of $G<\X(G)$ on $W$ by conjugation. We shall make heavy use 
of the notation and 
structure established in section \ref{s:structure}.

The first insight into the action of $Q$ on $W$ comes from Proposition \ref{l:nilp}: that 
proposition tells us that the action of $Q$ on the image of $W$ in $L/L'$ is nilpotent.
As in Lemma \ref{nilpotent12}, we have an exact sequence of $\Z Q$-modules
$$
H_2(L/W, \Z) \to W \to L/L' \to H_1(L/W, \Z) \to 0,
$$
so our main task now is to understand the structure of $H_2(L/W, \Z)$ as a $\Z Q$-module.

$L/W$ is the image of $L$ under $\rho: \X(G)\to G \times G\times {G}$, which
is
$$\langle (g, 1,{g}^{-1}) \mid g \in G \rangle =  \{ (g_1,1, {g}_2 ) \mid  g_1 g_2 \in G' \}.$$

Note that $L/W$ is normal in $G\times 1 \times {G}\cong G\times G$ with quotient $(Q\times Q)/Q_0$, where
$$
Q_0 = \{ (q,q^{-1}) \mid q\in Q\} < Q\times Q.
$$
To lighten the notation, we drop the second coordinate and identify $L/W$
with 
$$S := \{ (g_1, {g}_2 ) \mid  g_1 g_2 \in G' \}< G\times {G}.$$

The $G$-action on $L/W$ coincides
with the action of the first factor $G\times 1 <G\times {G}$ by conjugation on $S$,
so it is the induced $Q$-action on the homology of $S$ that we must understand.
(Note that this is not the same as the action of $Q_0$.) 

\begin{proposition} \label{action} Let $G$ be a finitely generated group, 
	let $Q$ be the abelianisation of $G$, let $S= {{ }  \langle \{(g,g^{-1}) \mid g\in G\} \rangle }<G\times G$ 
	and consider  the action of $Q$ on the homology of $S$ that is
	induced by the conjugation action of $G\times 1$. Then, 
	$J := H_2(S, \mathbb{Z})$ 
	has a filtration by  $\mathbb{Z} Q$-submodules $$J_1 \subseteq J_2 \subseteq J_3 \subseteq J$$  such that:
	\begin{enumerate}
		\item $J_1$ is a $\mathbb{Z} Q$-subquotient of the module $M=((G'/G'') \otimes_{\mathbb{Z}} (G'/ G''))_{Q_0}$ 
		where $Q = G / G'$ is acting by conjugation on the first factor 
		and acting trivially on the second factor, while the co-invariants
		are taken with respect to the standard conjugation action of
		$Q_0 = \{ (q, q^{-1}) \mid q \in Q \} < (G/G') \times (G/G')$;
		\item $Q$ acts trivially on  $J/ J_3$,  $J_3 / J_2$ and $J_2 / J_1$.
	\end{enumerate}
\end{proposition} 

\begin{proof}
	We shall write $\-{G}$ for the second direct
	factor of $G\times G$ in order to simplify the notation. Thus  $S<G\times\-{G}$.
	
	We analyse the LHS spectral sequence associated to the short exact sequence
	$1\to G'\times \-{G}'\to S\to Q_0\to 1$. Thus we consider
	$$E_{p,q}^2 = H_p(Q_0, H_q(G'\times \-{G}', \mathbb{Z}))$$ converging to $H_{p+q}(S, \mathbb{Z})$. 
	The terms $E^\infty_{i,j}$ with $i+j=2$ give a filtration of $H_2(S,\Z)$, so our focus  is on the terms $E^*_{i,j}$ with $i+j=2$. 
	
	The action of $G=G\times 1$ by conjugation (which determines
	the $Q$ action) is compatible with
	all of the maps and decompositions that we consider. The action of $G$ on 
	$E_{2,0}^2 = H_2(Q_0, \mathbb{Z})$ factors through the conjugation action of $Q\times 1$ on $Q_0$, which is trivial since $Q \times Q$ is abelian. Thus
	\begin{equation} \label{action1}
	\hbox{ the action of }Q \hbox{ on }E_{2,0}^{\infty} \hbox{ is trivial.}
	\end{equation}
	Next we consider 
	\begin{equation} \label{spectral1} 
	\begin{split}
	E_{1,1}^2 & = H_1(Q_0, H_1(G' \times \-{G}', \mathbb{Z})) =
	H_1(Q_0, (G' / G'') \times (\-{G}' / \-{G}'') ) \\
	& = H_1(Q_0, G' / G'') \oplus H_1(Q_0, \-{G}' / \-{G}''),
	\end{split}
	\end{equation} 
	The action of $G=G\times 1$  on $Q_0$
	and  $\-{G}' / \-{G}''$ is trivial, 
	hence $G$ acts trivially on the second factor of (\ref{spectral1}), i.e. on $H_1(Q_0, \-{G}' / \-{G}'')$.
	The action
	of $(h,1)\in G\times 1$ on  the coefficient module $G' / G''$ is the same as the 
	action of $(h, \-{h}^{-1}) \in S$ induced by conjugation in $G\times\-{G}$. 
	But $S$ acts trivially on  $H_1(Q_0, G' / G'')$, because
	the $S$-action factors through $Q_0$.  
	Thus the action of $G$ on the first
	factor of  (\ref{spectral1}) is also trivial, and
	\begin{equation}\label{action2} 
	Q \hbox{ acts trivially on } E_{1,1}^{\infty}.
	\end{equation}
	
	Finally we consider $E_{0,2}^2$. From the K\"unneth  formula, we have
	the $Q$-invariant decomposition of $Q_0$-modules
	$$
	H_2(G' \times \-{G}', \mathbb{Z}) \cong 
	H_2(G' , \mathbb{Z}) \oplus
	H_2(\-{G}', \mathbb{Z}) \oplus 
	H_1(G', \mathbb{Z}) \otimes H_1(\-{G}', \mathbb{Z}).
	$$
	The term $E_{0,2}^2$ is obtained by taking $Q_0$ co-invariants
	$H_0(Q_0, -)$ of these modules. 
	The argument used in our analysis of $E_{1,1}^2$ shows that
	$Q $ acts trivially on 
	$H_0(Q_0, H_2(G', \mathbb{Z})) \oplus H_0(Q_0, H_2(\-{G}', \mathbb{Z}))$,
	so 
	\begin{equation} \label{nov2}
	E_{0,2}^2.\, \A (\mathbb{Z} Q) \subseteq H_0(Q_0, H_1(G', \mathbb{Z}) \otimes H_1(G',\mathbb{Z})).
	\end{equation} 
	And by definition, 
	\begin{equation}\label{nov3} 
	H_0(Q_0, H_1(G', \mathbb{Z}) \otimes H_1(G',\mathbb{Z})) =    
	((G' / G'') \otimes (\-{G}' / \-{G}''))_{Q_0}
	\end{equation}
	is the quotient of $(G'/ G'') \otimes (\-{G}'/ \-{G}'')$  by the conjugation
	action of $Q_0< (G/G')\times (\-{G}'/\-{G}'')$. This is the module called
	$M$ in the statement of Theorem \ref{propB}.
	
	The spectral sequence converges to $H_*(S,\Z)$,
	so there is  a filtration $\{ F^j\}_j$ of $H_2(S, \mathbb{Z})$ such that
	$$
	F^{-1} = 0 \subseteq F^0 \subseteq F^1 \subseteq F^2 = H_2(S, \mathbb{Z}),
	$$
	with $F^i / F^{i-1} \cong E^{\infty}_{i,2-i}$.
	We saw in (\ref{action1}) and (\ref{action2}) that $Q$ acts trivially on $F^2/ F^1$ and $F^1/ F^0$.
	We set $$J_2 := E_{0,2}^{\infty} = F^0, \ 
	J_1 := J_2. \, \A (\mathbb{Z} Q), \ J_3 := F^1.$$ By (\ref{nov3}) and (\ref{nov2}), $E_{0,2}^2 \, \A (\mathbb{Z} Q)$ is isomorphic to a $\mathbb{Z} Q$-submodule of 
	$M=((G' / G'') \otimes (\-{G}' / \-{G}''))_{Q_0}$. Thus 
	$J_1$ is a $\Z Q$-subquotient of $M$.  
\end{proof}

\subsection{Proof of Theorem \ref{propB}} 
We return to consideration of the
following exact sequence of $\Z Q$-modules, 
$$
H_2(L/W, \Z) \overset{\mu}\to W \to L/L' \to H_1(L/W, \Z) \to 0.
$$
The first map has image $W\cap L'$, so from the filtration
$J_1 \subseteq J_2 \subseteq J_3 \subseteq J$ of Proposition \ref{action} 
we obtain a filtration 
$$\mu(J_1) \subseteq \mu(J_2) \subseteq \mu(J_3) \subseteq W_1:=W\cap L'$$
by $\Z Q$-submodules, such that $\mu(J_1)$ is a $\Z Q$-subquotient  
of $M=((G'/G'') \otimes_{\mathbb{Z}} (G'/ G''))_{Q_0}$, where the structure of $M$ as a $\Z Q$-module is
as described in the statement of Proposition \ref{action}, and 
$Q$ acts trivially on each of  $W_1/ \mu(J_3),\ \mu(J_3) / \mu(J_2)$ and 
$\mu(J_2) / \mu(J_1)$. In particular, defining $W_0:=\mu(J_1)$,
we have that $Q$ acts nilpotently on $W_1/W_0$, indeed $W_1. (\A (\Z Q))^3 \subseteq
W_0$.  

We proved in Proposition \ref{l:nilp} that $Q$ acts nilpotently on 
$L  / L'$, which contains $W/W_1$. Thus, $Q$ acts nilpotently on $W/W_0$
and the theorem  is proved.
\qed

\medskip
The following consequence of Theorem \ref{propB} will play a vital role in the proof of Theorems \ref{t:dehn}
and \ref{grw}.

\begin{corollary}\label{action21}  Let $G$ be a finitely generated group 
	such that $G /  G''$ is virtually nilpotent. Then there is a subgroup $Q_1$ of finite index in $Q = \X(G) /  DL \cong G /  G'$ and a filtration
	of  $W = W(G)$ by $\mathbb{Z} Q_1$-submodules such that  $Q_1$ acts trivially on each quotient of the filtration that is infinite.
\end{corollary}

\begin{proof} Maintaining the notation of the preceding proof, 
	we have $N:= W. (\A(\mathbb{Z} Q))^{3+s} \subseteq W_0$ where $s$ is the nilpotency class of the action of $Q$ on $L / L'$, i.e. $(L / L')(\A(\Z Q))^s = 0$ but $(L / L')(\A(\Z Q))^{s-1} \not= 0$.  Write $A$ for $G'/G''$.

	Let $T$ be a subgroup of finite index in $G$ such that $T / G''$ is nilpotent. Set $A_1 = (T/ G'') \cap A$ and $Q_1 = T G'/ G'$. Then $Q_1$ has finite index in $Q$, $A_1$ has finite index in $A$, and $Q_1$ acts nilpotently on $A_1$. 
	Let us say that a filtration of $\Z Q_1$-submodules is {\em{``good"}} if $Q_1$ acts trivially on every infinite quotient of the filtration.
	Note that the image $M_1$ of $A_1 \otimes_{\mathbb{Z}} A_1$ in $M= (A \otimes_{\Z} A)_{Q_0}$ has finite index and $M_1$ is a $\Z Q_1$-submodule of $M$. Moreover, since $Q_1$ acts nilpotently on  $A_1$ and $M_1$ has finite index in $M$, the filtration
	$
	0 \subseteq M_1 \subseteq M
	$ can be refined to 
	a good filtration of $M$. Since $N$ is a $\Z Q$-subquotient of $M$, it is a $\Z Q_1$-subquotient of $M$, and so $N$ has a good filtration by $\Z Q_1$-submodules. Finally, since $Q_1$-acts nilpotently on $W / N$, there is a good filtration of $\Z Q_1$-submodules of $W/ N$. These two good filtrations yield a good filtration of $W$
	by $\Z Q_1$-submodules.
\end{proof}

We also highlight the special case of Theorem C in which $G'/G''=0$.

\begin{cor}\label{action210}  If $G$ is finitely generated with  perfect commutator $G'$, 
	then  $Q$ acts nilpotently on $W = W(G)$.
\end{cor} 

\section{Isoperimetric functions for $\X(G)$} \label{s:isoper}

In this section we prove Theorem \ref{t:dehn}. The proof relies on the understanding of $W$ as a  $\Z Q$-module that was developed in the previous
section as well as on previous results on the isoperimetric functions of subdirect products and central extensions of groups, and on the
basic facts about isoperimetric functions gathered in Section \ref{s:dehn}.

\subsection{Subdirect products that are co-abelian}  
Let $H$ be a subgroup of a group $G$. One says that $H$ is {\em coabelian in $G$} if $G' \leq H$.
If there exists a subgroup of finite index $G_0<G$ with $[G_0, G_0] \leq H$, then $H$ is said to be
{\em{virtually-coabelian in $G$}}. The {\em{corank}} of $H$ in $G$ is defined
to be $\dim_{\mathbb{Q}} (G_0 / (G_0 \cap H) \otimes_{\mathbb{Z}} {\mathbb{Q}})$. Note that
this is independent of the finite-index subgroup $G_0$ chosen.

Suppose now that $H$ is a subgroup of a direct product
$D= \Gamma_1 \times \Gamma_2 \times \ldots \times \Gamma_n$. If $\Gamma_i H = D$ for every $1 \leq i \leq n$, we say that $H$ 
{\em{fills}} $D$. If $[D : \Gamma_i H] < \infty$ for every $ 1 \leq i \leq n$ we say that $H$ is {\em{virtually filling}} in $D$.
Note that these definitions depend on the choice of direct-product decomposition of $D$. 

Our proof of Theorem \ref{t:dehn} relies on the following result from Will Dison's thesis \cite{dison}.

\begin{theorem}[{\cite[Thm.~A]{dison}}]\label{Disonthm} Let $H$ be a virtually-filling subgroup of a direct product
	$D= \Gamma_1 \times \Gamma_2 \times \ldots \times \Gamma_n$, and suppose that  $H$ is virtually-coabelian of  corank $r$.
	Suppose each $\Gamma_i$
	is finitely presented and $n \geq 3$. For each
	$i$, let $( \alpha_i, \rho_i)$ be an area-radius pair for some finite presentation of
	$\Gamma_i$. Define
	$$\alpha(l) = \max(
	\{l^2\} \cup \{ \alpha_i(l)  : 1 \leq i \leq n \})
	$$
	and
	$$\rho(l) = \max(
	\{ l \} \cup \{ \rho_i(l) : 1 \leq i \leq n \})
	.$$
	Then
	$\rho^{2r} \alpha$
	is an isoperimetric function for
	$H$.
\end{theorem}  

\subsection{Isoperimetric functions for central extensions}

Our proof of Theorem \ref{t:dehn} also relies on
the following  special case of the Bryant Park Lemma from \cite{B-B-M}.  

\begin{lemma} \label{central} Let $G$ be a finitely presented group and let $C$ be a finitely
	generated central subgroup. The Dehn functions of $G$ and  $G/ C$ satisfy the inequality $\delta_G(n)\preceq \delta_{G/C}(n)^2$.
	In particular, if $G/C$ satisfies a polynomial isoperimetric inequality, then so does $G$.
\end{lemma}

\begin{proof} In the light of  Lemma \ref{l:qi}, we may assume that $C$ is free abelian, of rank $k$ say.
	
	Let $\langle b_1, \ldots, b_m \mid r_1, \ldots, r_s \rangle$ be a finite presentation of $G/C$. We
	fix a basis  $\{a_1, \ldots, a_k \}$ for $C $ and choose words $\sigma_1, \ldots, \sigma_s$ in the letters $a_j^{\pm 1}$
	such that  $r_i \sigma_i = 1$ in $G$ for $ i=1,\dots,s$.  Let $\mu =\max_i |\sigma_i|$. We work with the following presentation of $G$: 
	$$
	\langle a_1, \ldots, a_k, b_1, \ldots, b_m \mid $$ $$r_i \sigma_i = 1,\, [a_t, a_j] = 1 = [a_t,b_d] \hbox{ for } 1 \leq i \leq s,\ 1 \leq  j,t \leq k,\   1 \leq d \leq m \rangle.
	$$
	Let $F$ be the free group on $\{ a_1, \ldots, a_k, b_1, \ldots, b_m \}$. For words $v_1, v_2\in F$, we write $v_1 =_c v_2$ if $v_1$ and $v_2$ represent the same element in $G$ and we can transform $v_1$ into $v_2$ ``at cost at most $c$", i.e. by applying at most  $c$ relations 
	from the above presentation of $G$; in other words ${\rm{Area}}(v_1v_2^{-1})\le c$.
	
	Let $w\in F$ be a word of length $|w|=n$ and suppose $w=1$ in $G$.  Using the relations $[a_t, b_d] = 1$ we have
	$$
	w =_{n^2} w_1(\underline{b}) w_2(\underline{a})
	$$
	where $w_2(\underline{a})$ is a word in the letters $\{ a_1^{\pm 1}, \ldots, a_k^{\pm 1} \}$ and  $w_1(\underline{b})$ is a word 
	in the letters $\{ b_1^{\pm 1}, \ldots, b_m^{\pm 1} \}$; both words have length at most $n$, and $w_1=1$ in $G/C$.
	Let $N={\rm{Area}}(w_1)$, which  is bounded by $\delta(n)$, where $\delta = \delta_{G/C}$ is the Dehn function of $G/C$. Then we have the following
	equalities, for some choice of relators $r_{j(i)}\in \{ r_1^{\pm 1}, \ldots, r_s^{\pm 1}\}$, where the conjugating elements $\theta_i$
	can  be taken to have length $|\theta_i|\le \delta(n)$, as in Lemma \ref{l:pair}:
	$$
	w_1(\underline{b}) = \prod_{1 \leq i \leq N} \theta_i^{-1} r_{j(i)} \theta_i = 
	\prod_{1 \leq i \leq N} \theta_i^{-1} (r_{j(i)} \sigma_{j(i)}) \sigma_{j(i)}^{-1} \theta_i =_{c}
	$$ $$    \prod_{1 \leq i \leq N} \theta_i^{-1} (r_{j(i)} \sigma_{j(i)})  \theta_i \sigma_{j(i)}^{-1}       =_N \prod_{1 \leq i \leq N} \sigma_{j(i)}^{-1} =: w_3(\underline{a})
	$$
	The first and second equalities are in the free group $F$ and the last is a definition. The cost $c$ of the third
	equality counts the number of commutation relations applied to move each  $\sigma_{j(i)}^{-1}$ past $\theta_i$: clearly, 
	$c \leq \sum_{1 \leq i \leq N} |\sigma_{j(i)}|.|\theta_i|\le \mu N\, \delta(n) \le \mu \delta(n)^2$.
	The cost of the fourth equality comes from applying the relations $r_j \sigma_j=1$ (followed by free reduction).
	Note that $w_3$ has length $|w_3|\le N\mu$ and hence $w_0:=w_3w_2$ has length at most $n + N\mu$. And $N\le \delta(n)$.
	
	At this stage, we have  
	$$w =_{n^2 + c + N} w_0(\underline{a})$$
	where $w_0$ is a word that represents the identity in $C= \<a_1,\dots,a_k\mid [a_i,a_j]=1, i,j=1,\dots,k\>$
	and  $|w_0| \le n + \mu \delta(n)$. The  Dehn function of $C$ is $q(m)=m^2$, 
	so $w_0 =_{M} 1$ where $M=(n + \mu \delta(n))^2$. Thus we have transformed $w$, which has length $|w|=n$, into the empty word at a total cost of at most
	$$
	n^2+ c + N + (n + \mu \delta(n))^2 \le n^2 + \mu \delta(n)^2 + \delta(n) + (n + \mu \delta(n))^2,
	$$
	and the last term is $\simeq \delta(n)^2$.
\end{proof}

The following corollary bends Lemma \ref{central} to the needs of Theorem \ref{t:dehn}.

\begin{corollary}\label{c:induct}
	Let $\G$ be a finitely presented group and let $A<\G$ be a normal subgroup that is finitely generated and abelian.
	Regard $A$ as a $\Z \G$-module via conjugation.
	Suppose that there is a filtration of $A$ by $\Z \G$-submodules $0=I_0<I_1<\dots < I_k=A$ such that $\G$ acts
	trivially on each of the quotients $I_j/I_{j-1}$ that is infinite.  Then the Dehn functions of $\G$ and $\G/A$
	satisfy $\delta_\G(n)\preceq \delta_{\G/A}(n)^{2^k}$.
	In particular, if $\G/A$ satisfies a polynomial isoperimetric inequality then so does $\G$.
\end{corollary}

\begin{proof} We proceed by induction on $k$, the length of the filtration. If $k=1$  then Lemma \ref{central} applies. In the
	inductive step we apply the case $k-1$ to $\G/I_1$, with  $A/I_1$ filtered by $(I_j/I_1)_j$,
	noting that $\G/A\cong (\G/I_1)/(A/I_1)$.
	
	If $I_1$ is finite, then the Dehn function of $\G/I_1$ is equivalent to that of $\G$, by Lemma \ref{l:qi}, and we are done.
	If  $I_1$ is infinite, then our hypothesis on the action of $\G$ implies that $I_1<\G$ is central.
	From the induction we know that $\delta_{\G/I_1}(n)\preceq \delta_{\G/A}(n)^{2^{k-1}}$,
	and   Lemma \ref{central} implies that $\delta_{\G}(n)\preceq \delta_{\G/I_1}(n)^2$.  This completes the induction.
\end{proof}

\subsection{Proof of Theorem \ref{t:dehn}}
For the convenience of the reader, we repeat the statement of Theorem \ref{t:dehn}.

\begin{theorem} \label{t:56}
	If $G$ is a finitely presented group whose maximal metabelian quotient $G/G''$ is
	virtually nilpotent, then there is a polynomial $p(x)$ such that 
	$$ \f_G(n) \preceq  \f_{\X(G)}(n) \preceq p\circ \f_G(n).$$
	In particular, $G$ satisfies a polynomial isoperimetric inequality if and only if
	$\X(G)$ satisfies a polynomial isoperimetric inequality (of different degree, in 
	general).
\end{theorem}

\begin{proof}   $G$ is a retract of $\X(G)$,  whence the leftmost inequality. 
	For the converse, we appeal to
	Corollary \ref{action21}, which provides us with 
	a subgroup $Q_1$ of finite index in $Q$ such that $W(G)$ has a filtration by $\Z Q_1$-submodules 
	$0=I_0<I_1 < I_2 <\dots < I_k=W$ such that $Q_1$ acts trivially on each infinite quotient $I_j/I_{j-1}$ of this filtration.  
	Let $\pi : \X(G) \to \X(G) / DL = Q$ be the canonical epimorphism. Then $T = \pi^{-1}(Q_1)$ has  finite index in $\X(G)$, so their
	Dehn functions are $\simeq$ equivalent.
	
	As $G/G''$ is finitely generated and virtually nilpotent, its subgroup $G'/G''$ is finitely generated, and therefore  $W$
	is finitely generated,  by Lemma \ref{l:KS}. Thus we can appeal to Corollary \ref{c:induct} with $W$ in the role of $A$
	and $T$ in the role of $\G$. From this we deduce that $T$ (hence $\X(G)$) will satisfy an isoperimetric
	inequality of the required form if $T/W$ does. But $T/W$ has finite index in $\X(G) / W$, and  
	$W$ is the kernel of the canonical map $\rho: \X(G) \to G\times G \times G$,
	whose  image is a  filling subdirect product that is normal with quotient $G/G'$.   Dison's theorem \ref{Disonthm} 
	tells us that the Dehn function of this subdirect product is bounded above by
	a polynomial function of the Dehn function of $G$. This completes the proof.
\end{proof}

\begin{cor}  Let $G$ be a finitely presented group whose commutator subgroup is perfect. 
	Then $\X(G)$ satisfies  a polynomial isoperimetric inequality if and only if $G$ does. 
\end{cor}

\begin{cor}  $\X(F)$ has a polynomial isoperimetric function, where $F$ is  Thompson's group.
\end{cor} 

\begin{proof} The commutator subgroup of $F$ is simple, and Guba \cite{Guba} proved that the Dehn function of $F$ is quadratic.
\end{proof}

\begin{example} The Dehn functions of semidirect products $G = \mathbb{Z}^k \rtimes_A \mathbb{Z}$ are completely understood. Bridson and Gersten \cite{BridsonGersten} 
	proved that $\delta_G(n)$ is determined by how the entries of the matrix $A\in{\rm{GL}}(k,\Z)$ grow under iteration:
	$\delta_G(n)\simeq n^2 |\!|A^n|\!|$. Thus $\delta_G(n)$ is polynomial if 
	$A$ is virtually unipotent (i.e. all of its eigenvalues are roots of unity)
	and is exponential otherwise. According to Theorem \ref{t:dehn}, $\X(G)$
	will satisfy a polynomial isoperimetric inequality if $G$ does, and in the other case
	the Dehn function of $\X(G)$ will be exponential. 
	In the polynomial case, $\X(G)$ is virtually nilpotent, by Theorem \ref{t:nilp}.  
	In all cases, $\X(G)$ is polycyclic \cite{LO}.
\end{example}

\section{$n$-Engel groups and $\X(G)$}

Recall that a group $G$ is {\em{$n$-Engel}} if $[a,b,\ldots, b] = 1$
for all $a,b\in G$,  where $b$ appears $n$ times in the left-normed commutator. 

\begin{theorem} If $G$ is a finitely generated $n$-Engel group, then $\X(G)$ is $m$-Engel 
	for $m = n+d+s+3$, where $d$ is the nilpotency class of $G/ G''$ and $s$ is the nilpotency class of the action of $G$ on $L/ L'$.
\end{theorem}

\begin{proof} The $n$-Engel words are defined inductively by 
	$\g_1(x,y)=[x,y]$ and $\g_{n+1}(x,y)=[\g_n(x,y),y]$.
	Let $a,b \in \X(G)$.
	Since $\X(G)/ W(G)$ is a subgroup of $G \times G \times G$, it is $n$-Engel, so
	$$
	\gamma_n(a,b) \in W(G).
	$$ We again consider $W = W(G)$ as a $\mathbb{Z} Q$-module via conjugation, where $Q=G/G'$. In the proof of Theorem \ref{propB}
	we constructed a $\mathbb{Z} Q$-submodule $W_0$  of $W$ such that  $W (\A \,\mathbb{Z} Q)^{3+s} \subseteq W_0$, where $\A \,\mathbb{Z} Q$ is the augmentation ideal of $\mathbb{Z} Q$ and $W_0$ is a $\mathbb{Z} Q$-subquotient of $M = ((G'/ G'') \otimes (G'/ G''))_{Q_0}$,
	where $Q=G/G'$ acts by conjugation on the first factor of the tensor product and trivially on the second factor.
	Thus
	$$
	\gamma_{n+s+3}(a,b) \in [[[W,b], \ldots],b] \in W(\A \,\mathbb{Z} Q)^{s+3} \subseteq W_0.
	$$
	Note that $W_0$ is a $\mathbb{Z} Q$-subquotient of $M$ and
	$M$ depends only on the metabelian group $G/ G''$. Karl Gruenberg \cite{gruen} proved that every finitely generated soluble $n$-Engel group is nilpotent, in particular $G/ G''$ is nilpotent, say of class $d$. It follows
	that $Q$ acts nilpotently on $A = G'/ G''$, more specifically $A ( \A \mathbb{Z} Q)^d = 0$. 
	From this it follows that  $Q$ acts nilpotently on $M = (A \otimes A)_{Q_0}$, 
	indeed  writing $\overline{a_1 \otimes a_2}$ for the image of $a_1 \otimes a_2 \in A \otimes A$ in $M$, for all
	$q_1, \ldots, q_d\in Q$ and $a_1,a_2 \in A$ we see that
	$$
	\overline{a_1 \otimes a_2}\,  (q_1-1)(q_2-1) \ldots (q_d - 1) = \overline{ a_1 (q_1-1) (q_2-1) \ldots (q_d-1) \otimes a_2}
	$$ 
	belongs to the image of $A ( \A \mathbb{Z} Q)^d \otimes A$ in $M$, which is trivial.
	Thus  $M ( \A\, \mathbb{Z} Q)^d = 0$, whence $ W_0 ( \A\, \mathbb{Z} Q)^d = 0$ and
	$$
	\gamma_{n+3+s+d}(a,b) \in W_0 ( \A\, \mathbb{Z} Q)^d = 0.
	$$
\end{proof}

\section{On the group ${\mathcal E}(G)$ }

Following Lima and Sidki \cite{LS}, define
$$
{\mathcal{E}}(G) = \langle G, \-G \mid [{\mathcal D}, {\mathcal L}] = 1 \rangle,
$$
where
$$
\mathcal{D} = {\mathcal{D}}(G) = [G,\-G], \ {\mathcal L} = {\mathcal L}(G) = \langle \{ g^{-1} \-g \mid g \in G \} \rangle, \ {\mathcal W} = {\mathcal W}(G) = {\mathcal D} \cap {\mathcal L}.$$
Each of these groups is normal in ${\mathcal E}(G)$ -- see \cite{LS}.
Consider the natural epimorphism
\begin{equation} \label{theta}
\theta : {\mathcal E}(G) \to \X(G),
\end{equation}
which restricts to the identity on $G \cup \-G$.
Then $D = D(G) = \theta({\mathcal D})$ and $L = L(G) = \theta(\mathcal{L})$, where $D$ and $L$ are as defined in introduction.
Note that ${\mathcal W} = \theta^{-1}(W(G))$ is central in $ {\mathcal D}{\mathcal L}$, and
$$
{\mathcal E}(G) / {\mathcal D}{\mathcal L} \cong \X(G)/ DL \cong G/G'
.$$

Kochloukova \cite{desi-EG} proved that the circumstances under which   ${\mathcal E}(G)$ is finitely presented are much more restricted
than for $\X(G)$.

\begin{theorem}[{\cite{desi-EG}}]\label{Ethm} ${\mathcal E}(G)$ is finitely presented if and only if $G$ is finitely presented and $G / G'$ is finite. In this case  ${\mathcal W}$  and ${\mathcal L}/ {\mathcal L}'$ are finitely generated.
\end{theorem}

\subsection*{Proof of Proposition \ref{p:EG}}
By Theorem \ref{Ethm}, if ${\mathcal E}(G)$ is finitely presented then $G / G'$ is finite, so $H:={\mathcal D}{\mathcal L}$
has finite index in ${\mathcal E}(G)$; in particular, $H$ is finitely presented and its Dehn function is $\simeq$ equivalent to that
of ${\mathcal E}(G)$. Moreover,  ${\mathcal W}$ is central in $H$, and by   Theorem \ref{Ethm} ${\mathcal W}$ is finitely generated. 
Lemma \ref{central} tells us that $\delta_H(n)\preceq \delta_{H/\mathcal{W}}(n)^2$. 
But $H/\mathcal{W}\cong DL/W$ is isomorphic to the image of $DL<\X(G)$ under $\rho:\X(G)\to G\times G \times G$, which
is normal with abelian quotient. As $G/G'$ is finite, the image of $\rho$ has finite index in  $G\times G \times G$, and hence 
its Dehn function is $\simeq$ equivalent to $\max \{ n^2, \delta_G(n)\}$.

\begin{remark}
	If $G$ is infinite, hyperbolic and perfect, then by arguing as in the proof of Proposition A, one can improve the bound in Proposition \ref{p:EG} and
	show that in this case $\delta_{\mathcal{E}(G)}(n)\simeq n^2$.
\end{remark}

\section{Solubility of the word problem in $\X(G)$} \label{s:solvWP}

We have seen that bounding the complexity of the word problem in $\X(G)$ by
means of reasonably efficient isoperimetric inequalities is a subtle challenge;
in particular it depends on more than just the Dehn function of $G$. In contrast,
the following theorem
shows that the mere existence of a solution to the word problem 
depends only on the existence of such an algorithm in $G$.
The proof of this result shows, roughly speaking, that
the complexity of solving the word problem  in $\X(G)$ is bounded by the greater of the
complexity of the word problem in $G$ and the complexity of the word problem in
the metabelian group $W(G)\rtimes Q$, where the action of $Q=G/G'$ is induced by
the conjugation action of $G$ in $\X(G)$.

\begin{theorem}\label{t:solv}
Let $G$ be a finitely presented group. The word problem in $\X(G)$ is soluble if and
only if the word problem in $G$ is soluble.
\end{theorem}

\begin{proof} {{ } Recall from \cite{BK1}, as in Section 1, that  if $G$ is finitely generated then $L$ is finitely generated and if $G$ is finitely presented then $\X(G)$ is finitely presented. Furthermore, if $G$ is finitely presented then by \cite{BHMS} $\imm(\rho) < G \times G \times G$ is finitely presented, so in particular $W$ is finitely generated as a normal subgroup of $\X(G)$.}

Any retract of a group with a soluble word problem has a soluble word
problem, so the real content of the theorem is the
``if" implication: we
assume that $G$ has a soluble word problem and must prove that $\X(G)$ does. 
Note that if $G$ has a soluble word problem, then so does $G\times G\times G$.

Once again we focus our attention on the subgroups
$W < L < \X(G)$ and the decomposition $\X(G)=L\rtimes G$ discussed in Section \ref{s:structure}, as well as the exact sequence 
$$
1\to W \to \X(G) \overset{\rho}\to G\times G \times G.
$$
We fix a finite generating set $X= A\cup B\cup C$ for $\X(G)$ with 
$A\subset W,\, B\subset L$
and $C\subset G$, where  $A$ generates $W$ as a normal subgroup,
$A\cup B$ generates $L$, and $C\subset G$ generates $G$.
We then fix a finite presentation $\< X\mid R\>$ for $\X(G)$.
We also fix a finite presentation $\<Y\mid S\>$ for $G\times G\times G$ 
where $X\subset Y$ generates the image of $\rho$ and 
$S$ includes both $R$ and $\{a : a\in A\}$.

Given a word $w$ in the free group on $X$, we determine whether or not it equals the
identity in $\X(G)$ by employing the following algorithm. First, working in the
free group $F=F(X)$, making repeated use of the identity $ua = au^a$, we
move all occurrences of letters $a_i\in A$ to the left, 
thus expressing $w\in F$ as a product $w=w_1w_2$, where $w_1$ is a word
in the letters $A$ and $w_2$ is a
product of conjugates of letters from $B\cup C$.

We can decide whether or not $w_2 =1$ in $\<Y\mid S\>$ using the
hypothesized solution to the word problem in $G\times G\times G$. If $w_2\neq 1$
 then $w\neq 1$ in $\X(G)$ and we are done. If 
$w_2 =1$ in $\<Y\mid S\>$, then the element of $\X(G)=\<X\mid R\>$
represented by $w_2$ lies in the kernel of $\rho$, which is $W$, and
therefore $w\in W$. We chose $A$ so that it generates $W$ as a normal
subgroup, so by searching na\"{i}vely through equalities in the free group $F$,
we will eventually find a product of conjugates of the letters $a\in A$
that equals $w_2$ in $\X(G)= \< X\mid R\>$; in other words, we find an equality
in $F$ of the from
$$
w_2 = \prod_{i=1}^N a_i^{u_i} \ \prod_{j=1}^M r_j^{\theta_j},
$$
where $a_i\in A^{\pm 1}$ and $r_j\in R^{\pm 1}$. Note that since $[L, W]=1$,
we may assume the conjugators $u_i$ are words in the letters $C^{\pm 1}$ alone
(since conjugation by $A\cup B$ has no effect in $W$). Define $w_2'$ to be the first
of the two products in this decomposition of $w_2$.

At this stage, we have transformed $w$ into a word $w':=w_1w_2'$  that represents
the same element of $W<\X(G)$ and is expressed in the free group $F(A\cup C)$
as a product 
$$w' = \prod_{i=1}^{N'} a_i^{v_i},$$
where $a_i\in A^{\pm 1}$ and the $v_i$ are words in the letters $C^{\pm 1}$. 
We must decide whether or not $w'=1$ in $W$.

The final key point to observe is that because the action of $G'$ by conjugation on 
$W <L$ in $\X(G)=L\rtimes G$ is trivial, the natural map $W\rtimes G\to W\rtimes (G/G')$
restricts to an injection on $W$. 
It follows that $w'=1$ in $W < \X(G)$
if and only if the image of $w'$ 
under the natural map $F(A\cup C)\twoheadrightarrow W\rtimes (G/G')$ is trivial. 
And we can decide if this last image is trivial because the word problem is
soluble in any finitely generated metabelian group -- this follows from
classical work of Philip Hall \cite{pHall1}, \cite{pHall2} who proved that finitely generated metabelian
groups are residually finite and finitely presented in the variety of 
metabelian groups. In fact, since
every finitely generated metabelian group
 can be embedded in a finitely presented metabelian group with
polynomial Dehn function \cite{kassabov}, such word problems lie in the complexity class {\rm{NP}}.
\end{proof}

\section{$\X$-Closure and Growth}\label{s:growth} \label{s:last}

In this section we explain how the structure that emerged in the proof of Theorem \ref{propB} leads to the closure criterion 
isolated in Theorem \ref{grw}, and we use this criterion to show that $\X$ preserves growth type (Corollary \ref{grw2}).

\subsection{Proof of Theorem \ref{grw}}
For any group $G$, we have surjections $\X(G)\twoheadrightarrow {\imm}(\rho_G)\twoheadrightarrow G$;
so if $\mathcal P$ is closed under quotients, then $\X(G)\in\mathcal{P}\Rightarrow {\imm}(\rho_G)\in\mathcal{P}$
and ${\imm}(\rho_G)\in\mathcal{P}\Rightarrow G\in \mathcal{P}$. We also have 
${\imm}(\rho_G) \le G\times G\times G$; so if $\mathcal{P}$ is closed under subgroups and 
finite direct products, then $G\in\mathcal{P}\Rightarrow {\imm}(\rho_G) \in \mathcal{P}$. With these
observations in hand, Theorem \ref{grw} and the variations on it stated in the introduction 
are immediate consequences of the following result. 

\begin{proposition} Let $G$ be a finitely generated group. If $G/G''$ is virtually nilpotent, then there is
a subgroup of finite index $H\le \X(G)$ that can be obtained from a finite-index subgroup of 
${\imm}(\rho_G)\le G\times G \times G$
by a finite sequence of extensions each of which has finite-abelian or finitely-generated-central kernel.
\end{proposition}

\begin{proof} As $G/G''$ is finitely generated and virtually nilpotent, the results in Section \ref{s:structures}
apply. In particular, Corollary \ref{action21} provides us with a subgroup $Q_1$ of finite index in 
$Q = \X(G)/ DL \cong G/ G '$  and  a finite filtration of $W$ by $\mathbb{Z} Q_1$-submodules 
 $$1 = W_0 \subseteq W_2 \subseteq \ldots \subseteq W_k = W$$
 such that $Q_1$ acts trivially on each of the 
quotients $W_i/ W_{i-1}$ that is infinite. Moreover, since $G'/G''\le G/G''$ is finitely generated, we know from
\cite{KochSidki} that each $W_i$ is finitely generated as an abelian group.
 
 Let $H\le\X(G)$ be the preimage of $Q_1$. This has finite index, so $H / W_k = H/ W$
 has finite index in $\imm(\rho_G)$. Consider the short exact sequence of groups
 $$1 \to W_k/ W_{k- 1} \to H/ W_{k-1} \to H/ W_k \to 1.$$
 There are two options: either $W_k/ W_{k-1}$ is finite or it is central in $H/ W_{k-1}$. 
 Repeating this argument with $k-1,\dots,0$ in place of $k$ completes the proof.
 \end{proof}
 
 \begin{remark}\label{r:vNilp}
 It is straightforward to verify that the class of finitely generated virtually nilpotent groups satisfies the version of 
 Theorem \ref{grw} in which only central extensions by finitely generated kernels are allowed. Thus we
 obtain a second proof of Theorem \ref{t:nilp} at the expense of using  a corollary of Theorem \ref{propB} in the new proof.  
 \end{remark}
 
 \subsection{Growth: Proof of Corollary \ref{grw2}}

Let $G$ be a group, let $S$ be a finite generating set for $G$ and let $d_S$ be the associated word metric.
The {\em{growth function}} ${\rm{vol}}_{G,S}$ counts the number of elements in balls about the identity in $G$:
$$
{\vol_{G,S}}(n) := |B_{G,S}(n)|,
$$
where $ B_{G,S}(n) =  \{g\in G \mid d_S(1,g)\le n\}$.

If there exist constants $C,\delta>0$ such that $\vol_{G,S}(n)\le C \, n^\delta$ for all $n>0$, then
$G$ is said to have {\em{polynomial growth}}. If $\lim_n {\vol_{G,S}}(n)^{1/n} >1$ then $G$
 has {\em{exponential growth}}, and if $\lim_n {\vol_{G,S}}(n)^{1/n} =1$
then $G$ has {\em sub-exponential growth}. If the growth of $G$ is sub-exponential but not polynomial, then $G$
is said to have {\em intermediate growth}.  It is easy to check that these growth types are
independent of the chosen generating set $S$. 

Our main interest in Theorem \ref{grw} lies in the following application, which is a restatement 
of Corollary \ref{grw2}. Note that item (2) of this proposition allows one to construct new groups of intermediate
growth $\X_n(G) = \X ( \X_{n-1}(G))$ starting from well-known examples such as the Grigorchuk group \cite{Gr}, 
 or the Gupta-Sidki groups \cite{G-S}. 

\begin{proposition} \label{p:grow}
Let $G$ be a finitely generated group. Then,
\begin{enumerate}
\item[$\bullet$]  $\X(G)$ has  polynomial growth if and only if $G$ has polynomial growth;
\item[$\bullet$] $\X(G)$ has subexponential growth if and only if $G$ has subexponential growth.
\end{enumerate}
\end{proposition}

\begin{proof}
Gromov \cite{Gromov} proved that groups of polynomial growth are virtually nilpotent, and the converse is
straightforward. Thus the first part of this result is covered by Theorem \ref{t:nilp} (alternatively, 
remark \ref{r:vNilp}).

Let $\mathcal{P}$ be the class of finitely generated groups of subexponential growth. It is easy to see that 
$\mathcal{P}$ is closed under quotients, subgroups, extensions by and of finite groups, and
finite direct products.
And it is well-known that if a soluble group is not  virtually nilpotent, then it
has exponential growth \cite{Milnor}, \cite{Wolf}; so the metabelian groups in $\mathcal{P}$ are virtually nilpotent.
Thus, before we can apply Theorem K, it only remains to check the closure of $\mathcal{P}$  under central extensions, which is the content
of the lemma that follows.
\end{proof}

\begin{lemma} \label{cent-subexp}
Let $G$ be a finitely generated group and let $Z<G$ be a finitely generated central subgroup.
If $G/Z$ has subexponential growth, then so does $G$.
\end{lemma}
 
 \begin{proof}   
 We fix a generating set $T=\{z_1,\dots,z_m\}$ for $Z$ and extend this to 
 a generating set $S=\{z_1,\dots,z_m, y_1,\dots, y_r\}$ for $G$. Let $Y=\{y_1,\dots,y_r\}$, let $H=\<Y\>$,
 and note that $HZ=G$. We write $\-{h}$ for the image of $h\in H$ in $\-{H}:=G/Z$.
 Note that $\-{H}$ is generated by $\-{Y}=\{\-{y}_1,\dots,\-{y}_r\}$.

As $d_S(1,hz) \le d_Y(1,h)+ d_T(1,z)$ for all $h\in H,\, z\in Z$, 
 $$ B_{G,S}(n) \subseteq B_{H,Y}(n) B_{Z,T}(n).$$
As $Z$ is abelian, $|B_{Z,T}(n)|\le C n^\delta$ for some constant $C$, where $\delta$ is the torsion-free rank of $Z$.
Thus  it suffices to prove that $H$ has sub-exponential growth.

 We define
 a set-theoretic section of $H\to\-{H}$ by choosing 
 coset representatives $\sigma(\-{h})\in H$ for $Z$ in $HZ$.  
 As $\-{H}$ has sub-exponential growth, 
 for every $\e>0$ there exists $N_\e>0$ such that ${\vol_{\-{H},\-{Y}}}(n) \le e^{\e n}$ for all
 $n\ge N_\e$.  Let $\zeta_h:=h\sigma(\-{h})^{-1}\in Z$  
 and let
$$\lambda_\e = \max\{  d_T(1,\zeta_h) \mid h\in H,\, d_Y(1,h)\le N_\e\} .$$

 For all integers $n>0$ and all $h\in H$ with $d_Y(1,h)\le n$,  if $k=\lceil n/N_\e\rceil$ then
 $h= h_1\dots h_k$ for some $h_i\in   B_{H,Y}(N_\e)$, therefore
 $$
 h = \prod_{i=1}^k \zeta_{h_i} \sigma(\-{h}_i) =  \prod_{i=1}^k \zeta_{h_i} \ \prod_{i=1}^k \sigma(\-{h}_i).
 $$
 There are at most $|B_{\-{H},\-{Y}}( N_\e)|^{k}$ possible values for the product of the $\sigma(\-{h}_i) $ and
 at most $C (k \lambda_\e)^\delta$ values for the product of the $\zeta_{h_i}\in Z$. And
 $$|B_{\-{H},\-{Y}}(N_\e)|^{k} =  {\vol_{\-{H},\-{Y}}}(N_\e)^k  \le e^{\e\, kN_\e} .$$
Therefore,
 $$|B_{H,Y}(n)|  \le   e^{\e\, kN_\e}\  C (k \lambda_\e)^\delta.$$
 Noting that $kN_\e < n+N_\e$, we deduce that
 $$
 \lim_{n\to\infty} \frac{1}{n} \log |B_{H,Y}(n)|   \le \e.
 $$
 And since $\e>0$ is arbitrary, we conclude that $\lim_n ({\vol_{H,Y}}(n))^{1/n}=1$, as required.
 \end{proof}

 \begin{remark}
 Tianyi Zheng \cite[Lemma 1.1]{zheng} proved that 
the preceding  lemma remains true even if  $Z$ is not finitely generated. 
 \end{remark}

 \subsection{The fixed point property for cones}

We close with a brief discussion of two  classes of groups  related to sub-exponential growth and amenability. 
We shall be concerned only with finitely generated groups, although the theory is more general.
Monod \cite{Monod} defines a group to have {\em the fixed point property for cones} if
actions of $G$  on convex cones in topological vector spaces must, under specified mild hypotheses, have fixed points. 
He proves that it is equivalent to require that for every non-zero  bounded function  $f$ on $G$ 
and all $t_i\in \mathbb{R}, \ g_i\in G$, if $f\ge 0$ then 
$$
\sum_{i=1}^n t_ig_if \ge 0 \implies \sum_{i=1}^n t_i \ge 0.
$$  
Following Rosenblatt \cite{rosen},  one says that a group $G$ is {\em supramenable} if  for every non-empty 
$A \subset G$ there is an invariant, finitely additive measure $\mu$ on $G$ such that $\mu(A) = 1$.
Kellerhals, Monod and  R\"{o}rdam, proved that $G$ is supramenable if and only if the rooted binary tree cannot be Lipschitz-embedded in
$G$  -- see \cite[Prop. 3.4]{K-M-R}.  Rosenblatt \cite{rosen}, had earlier proved that a supramenable group cannot contain
a non-abelian free semigroup.

Groups of sub-exponential growth have the  fixed point property for cones, and if a group has the fixed point property for
cones then it is supramenable \cite{Monod}. It is unknown if these implications can be reversed for finitely
generated groups. It is also unknown whether either of the 
latter classes is closed under the formation of finite direct products.

 \begin{proposition} 
 Let $G$ be a finitely generated group.  $\X(G)$ has the fixed point property
 for cones 
 if and only if ${\imm}(\rho_G)=\{(g_1,g_2,g_3) \mid g_1g_2^{-1}g_3\in [G,G]\}<G\times G\times G$ 
 has the same property.
 \end{proposition}
 
 \begin{proof} It suffices to prove that the class of groups with the fixed point property for cones 
 satisfies each of the criteria in Theorem \ref{grw} except closure under direct products.
 Closure
 under the operations of forming
 subgroups,  quotients and  central extensions,  as well as extensions by and of finite groups is proved by Monod in \cite{Monod}.
 Thus it only remains to  show that a finitely generated metabelian group $G$ with the fixed-point property for cones is virtually nilpotent. 
 As discussed above, $G$ {{ } is }supramenable, so it cannot contain a non-abelian free semigroup. 
Rosenblatt \cite{rosen} proved that a finitely generated solvable group without such a semigroup must have polynomial growth,
and it is therefore virtually nilpotent \cite{Milnor}, \cite{Wolf}.
\end{proof}

\begin{remark} If the class $\mathcal P$ of finitely generated groups with the fixed-point property for cones were closed under the taking of finite direct products, we would be able to reverse the implication  ${\imm} (\rho_G) \in {\mathcal P}\implies G \in {\mathcal P}$.
\end{remark}


\begin{thebibliography}{99}
	
	
	\bibitem{AO} 
	Arzhantseva, G.N.; Osin, D.V.,
	{\em Solvable groups with polynomial Dehn functions}, Trans. Amer. Math. Soc. {\bf 354} (2002),   3329--3348. 
	
	
	
	\bibitem{B-B-M} Baumslag, G.; Bridson, M.R.; Miller III, C.F.,
	{\em Isoperimetric and isodiametric functions of group extensions}, in
	``Elementary Theory of Groups and Group Rings, and Related Topics" (P.~Baginski {\em et al.}, eds.), De Gruyter, Berlin, 2020,
	pp.~7--10. 
	
	\bibitem{bieri} Bieri, R., 
	{\em Homological dimension of discrete groups}, Second edition. Queen Mary College Mathematical Notes. Queen Mary College, Department of Pure Mathematics, London, 1981.
	
	
	
	\bibitem{bfs} Bridson, M.R.,  
	{\em The geometry of the word problem}, 
	in ``Invitations to geometry and topology",
	(M.R. Bridson and S.M. Salamon, eds.), OUP, Oxford 2002,
	pp.~29--91.
	
	
	
	\bibitem{mrb:camb} Bridson, M.R., {\em Semihyperbolicity},  in ``Beyond Hyperbolicity" (M.~Hagen, R.~Webb, H.~Wilton, eds.), London Math Soc Lecture Note Series {\bf 454}, Camb. Univ. Press, Cambridge, 2019, pp.~25--64.
	
	
	
	
	\bibitem{BridsonGersten}  Bridson, M.R.; Gersten, S.M., 
	{\em The optimal isoperimetric inequality for torus bundles over the circle},
	Quart. J. Math. Oxford Ser. (2) {\bf 47} (1996), no. 185, 1--23.
	
	
	
	\bibitem{BH}
	Bridson M.R.; Haefliger, A.,
	``Metric Spaces of Non-Positive Curvature", Grund. Math. Wiss.
	{\bf 319}, Springer-Verlag, Heidelberg-Berlin, 1999.
	
	
	
	\bibitem{BHMS} Bridson, M.R.; Howie, J.; Miller III, C.F.; Short, H.,
	{\em On the finite presentation of subdirect products and the nature of residually free groups}, Amer. J. Math. {\bf 135} (4) (2013), 891 -- 933.
	

\bibitem{kassabov} Bridson, M.R.; Kassabov, M.;
Kochloukova, D.H.; Matuccci, F, {\em Isoperimetric functions and 
embedding properties of metabelian groups}, in preparation.	
	
	\bibitem{BK1} Bridson, M.R.; Kochloukova, D., {\em Weak commutativity 
		and finiteness properties of groups}, Bull. London Math. Soc. {\bf{51}} (2019), 168--180.
	
	
	\bibitem{brown} Brown, K.S. {\em Cohomology of groups}, Graduate Texts in Mathematics, 87. Springer-Verlag, New York, 1994
	
	
	
	\bibitem{B-G} Brown,  K.S.;  Geoghegan, R., {\em An infinite-dimensional torsion-free ${\rm FP}\sb{\infty }$ group}, Invent. Math. {\bf 77}(1984),   367--381.
	 
	
	
	\bibitem{delz} Delzant, T.,
	{\em L'invariant de Bieri-Neumann-Strebel des groupes fondamentaux des vari\'{e}t\'{e}s k\"{a}hl\'{e}riennes,}
	Math. Ann. {\bf{348}} (2010),   119--125. 
	
	\bibitem{dison} Dison, W., {\em Isoperimetric functions for subdirect products and Bestvina-Brady groups}, PhD thesis, Imperial College London, 2008. arXiv:0810.4060
	
	
	\bibitem{epstein} Epstein, D.B.A.; Cannon, J.W.; Holt, D.F.; Levy, S.V.F.; Paterson, M.S.; Thurston, W.P.,
	{\em Word processing in groups}, Jones and Bartlett Publishers, Boston, MA, 1992.
	
	\bibitem{ershov} Ershov, M.; He, S., {\em On finiteness properties of the Johnson filtrations},
	Duke Math. J. {\bf 167} (2018),  1713--1759.  
	
	
	\bibitem{gersten} 
	Gersten, S.M.,  
	{\em Isoperimetric and Isodiametric Functions of Finite Presentations},
	In  ``Geometric Group Theory" (G. Niblo and M. Roller, Eds.), London Math. Soc. Lecture Notes {\bf 181}, 1983, Cambridge Univ. Press, Cambridge, pp. 79--96.
	
	
	
	\bibitem{GS} Gersten, S.M.; Short, H.B., {\em Rational subgroups of biautomatic groups},
	Ann. of Math. {\bf{134}} (1991),  125--158. 
	
	\bibitem{Gr} Grigorchuk, R. I., {\em On the Milnor problem of group growth,} Soviet Math. Dokl. {\bf 28} (1983), 23 -- 26.
	
	 
	
	\bibitem{Gromov} Gromov, M., {\em Groups of polynomial growth and expanding maps,} Publ. Math. IHES {\bf 53} (1981), 53--78. 
	
	\bibitem{gruen} Gruenberg K.W., {\em Two theorems on Engel groups},
	Proc. Camb. Phil. Soc. {\bf 49} (1953), 377--380.
	
	\bibitem{Guba} Guba, V.S., {\em The Dehn function of Richard Thompson's group F is quadratic}, Invent. Math. {\bf 163} (2006),   313--342.
	
	
	\bibitem{GubaSapir} Guba, V.S.; Sapir, M.V., {\em The Dehn function and a regular set of normal forms for R. Thompson's group F},
	J. Austral. Math. Soc. Ser. A {\bf 62} (1997), 315--328.
	
	\bibitem{G-S} Gupta, N., Sidki, S. {\em On the Burnside problem for periodicgroups,} Math. Z., {\bf 182} (1983), 385--388.
	
	\bibitem{G-R-S}
	Gupta, N.; Rocco, N.; Sidki, S., {\em Diagonal embeddings of nilpotent groups},
	Illinois J. Math. {\bf 30} (1986),  274--283.
	
	\bibitem{Hall} Hall, M. Jr., {\em Subgroups of finite index in free groups}, Canad. J. Math. {\bf 1} (1949), 187--190.
	
	\bibitem{pHall1}
Hall, P., {\em Finiteness conditions for soluble groups}, Proc. London Math. Soc. {\bf 4}, (1954), 419--436. 

\bibitem{pHall2}
Hall, P., {\em On the finiteness of certain soluble groups}, Proc. London Math. Soc. {\bf 9} (1959), 595--622.
	 
	 \bibitem{K-M-R} Kellerhals, J., Monod, N., R\"{o}rdam, M., {\em Non-supramenable groups acting on locally compact
	 spaces,} Doc. Math. {\bf 18} (2013), 1597--1626.
	 
	
	\bibitem{desi-EG} Kochloukova, D.H.,  {\em Homological and homotopical finiteness properties of the group $E(G)$},
	Monatsh. Math. {\bf 186} (2018),   93--109. 
	
	\bibitem{KochSidki} Kochloukova, D.; Sidki, S., 
	{\em On weak commutativity in groups}, J. Algebra {\bf 471} (2017), 319--347. 
	
	\bibitem{LO} Lima, B.C.R.;  Oliveira, R.N, {\em Weak commutativity between two isomorphic polycyclic groups},  J. Group Theory, {\bf 19}  (2016),  239--248.
	
	\bibitem{LS}
	Lima, B.C.R.;  Sidki, S.,
	{\em A generalization of weak commutativity between two isomorphic groups},
	J. Group Theory {\bf 2} (2017),  775--792.  
	
	\bibitem{Milnor} Milnor, J., {\em Growth of finitely generated solvable groups}, J. Diff. Geom. {\bf 2} (1968), 447--449. 
	
	\bibitem{Monod} Monod, M., {\em Fixed points in convex cones}, Trans. Amer. Math. Soc. Ser. B, 4 (2017), 68--93. 
	
	\bibitem{mosher} Mosher, L., {\em Central quotients of biautomatic groups}, Comment. Math. Helv. {\bf 72} (1997),  16--29.
	
	
	
	\bibitem{NR} Neumann, W.D.; Reeves, L., {\em Central extensions of word hyperbolic groups},
	Ann. of Math. {\bf 145} (1997), 183--192. 
	
	\bibitem{rips} Rips, E., {\em Subgroups of small cancellation groups}, Bull. London Math. Soc. {\bf 14} (1982),  45--47.
	
	\bibitem{rosen}
	Rosenblatt, J.M.,  {\em Invariant measures and growth conditions},
	 Trans. Amer. Math. Soc. {\bf 193} (1974), 33--53.
	
	\bibitem{Said}  Sidki, S., {\em On weak permutability between groups}, J. Algebra 
	{\bf 63} (1980),   186--225.
	
	\bibitem{Wolf} Wolf, J., {\em Growth of finitely generated solvable groups and curvature of Riemanniann manifolds,} J. Diff. Geom, {\bf 2} (1968), 421--446. 
	
	\bibitem{zheng} Zheng, T., {\em On FC-central extensions of groups of intermediate growth,} arXiv 2001.07814.pdf 
	
	
\end{thebibliography}
\end{document}